\documentclass{article}
\usepackage {amssymb}
\usepackage {amsmath}
\usepackage {amsthm}
\usepackage{graphicx}
\usepackage {amscd}
\usepackage{xypic}
\usepackage[colorlinks, citecolor=red]{hyperref}
\newcommand{\Fut}{{\rm DF}}

\newcommand{\bX}{{\bar{\mathcal{X}}}}
\newcommand{\bL}{{\bar{\mathcal{L}}}}
\newcommand{\mH}{{\mathcal{H}}}
\newcommand{\mY}{{\mathcal{Y}}}
\newcommand{\mL}{{\mathcal{L}}}
\newcommand{\mX}{{\mathcal{X}}}
\newcommand{\bM}{{\bar{\mathcal{M}}}}

\newcommand{\bPi}{{\bar{\pi}}}
\newcommand{\mop}{{\mathcal{O}_{\mathbb{P}^1}(1)}}
\newcommand{\mot}{{\mathcal{O}_{\mathbb{P}^1}(-1)}}
\newtheorem{thm}{Theorem}
\newtheorem{prop}{Proposition}
\newtheorem{defn}{Definition}
\newtheorem{cor}{Corollary}
\newtheorem{rem}{Remark}
\newtheorem{conj}{Conjecture}
\newtheorem{exmp}{Example}
\newtheorem{lem}{Lemma}
\newtheorem{claim}{Claim}

\begin{document}

\title{Special test configuration and K-stability of Fano varieties}
\author{Chi Li, Chenyang Xu}
\date{\today}
\maketitle{}

\begin{abstract}

For any flat projective family $(\mX,\mL)\rightarrow C$ such that
the generic fibre $\mX_\eta$ is a klt $\mathbb{Q}$-Fano variety and
$\mL|_{\mX_\eta}\sim_{\mathbb{Q}}-K_{X_{\eta}}$, we use the techniques from the minimal model program
(MMP) to modify the total family. The end
product is a family such that {\it every} fiber is a klt
$\mathbb{Q}$-Fano variety. Moreover, we can prove that the
Donaldson-Futaki invariants of the appearing models  {\it decrease}. When the family is a test configuration of a fixed Fano
variety $(X,-K_X)$, this implies Tian's conjecture: given $X$ a
Fano manifold, to test its K-(semi, poly)stability, we only need to test on the special
test configurations. 

\end{abstract}

\begingroup
\hypersetup{linkcolor=blue}
\tableofcontents
\endgroup

\bigskip

The paper is motivated by studying Tian's conjecture which says that to test K-(semi, poly)stability we only need to consider the test configurations whose special fibers are $ \mathbb{Q}$-Fano varieties. It  consists of two parts. In the first part, inspired by Tian's conjecture we obtain our main result on the existence of special degenerations of Fano varieties. In the second part, we apply the result from the first part to study K-stability of Fano varieties. In particular, we give an affirmative answer to Tian's conjecture. 
\part{Family of Fano Varieties}\label{part1}

Through out this part, we work over an algebraically closed field which is of characteristic 0.

\section{Introduction: main results}

\subsection{Degenerations of Fano varieties}
For various questions, especially for compactifying the moduli spaces, people are interested in the degenerations of varieties. When the varieties  have positive canonical classes, i.e., they are {\it  canonically polarized}, this question has attracted many interests. The one dimensional case, namely, the degeneration of smooth curves of genus at least 2, has been understood well after Deligne-Mumford's work \cite{DM69}.
The study of higher dimensional case by an analogous strategy  was initiated more than two decades ago ( see \cite{KSB}, \cite{Ale94}). Using the recent monumental progress on the minimal model program of \cite{BCHM} and many other work, the fundamental aspects of this theory are close to be completely settled. See Koll\'ar's survey paper \cite{Kol} for more details.  One essential point of such varieties having a nice moduli theory is that they carry  natural polarizations, namely, their canonical classes.

Another class of varieties carrying natural polarization is the
class of Fano varieties, whose canonical classes are negative. However,
such varieties behave quite differently with the canonically
polarized varieties. For example, there exists a family of Fano
manifolds, whose general  fibers are isomorphic to a given Fano
manifold, but the complex structure  jumps at the special fibers
(see e.g. \cite[Section 7]{Tian97}, \cite[2.3]{PP10}). This means that  the functor of Fano
manifolds in general is not separated. Nevertheless, even without
knowing the uniqueness we can still ask generally whether a `nice'
degeneration exists. Of course, this depends on the meaning of `nice'. In this paper, we are looking for  degenerations satisfying two properties: 

First, the degenerate fibers should be mildly singular and still with negative canonical classes. Recall that a
variety $X$ is called a {\it $\mathbb{Q}$-Fano variety} if it only
has klt singularities (see \cite{KM} for the meaning of the
terminology) and $-K_X$ is ample.   In particular, a
$\mathbb{Q}$-Fano variety is normal. This class of varieties plays a
central role in birational geometry. From many viewpoints, it has a
similar behavior as Fano manifolds. 

\begin{defn}Let $f:(\mX,\mL)\rightarrow C$ be a flat projective morphism over a smooth curve, where $\mX$ is normal and $\mL$  is an $f$-ample $\mathbb{Q}$-line bundle. We call $(\mX,\mL)\to C$ a polarized generic  $\mathbb{Q}$-Fano family if there exists a non-empty open set $C^*\subset C$, such that  for any $t\in C^*$, $ \mX_t $ is klt and $\mL|_{\mX^*}\sim_{\mathbb{Q}, C^*}-K_{\mX^*}$.  In this case, we say $C^*$ parametrizes non-degenerate fibers. If we can choose $C^*=C$, then we  call $f:(\mX,-K_{\mX/C})\rightarrow C$ a $\mathbb{Q}$-Fano family,

\end{defn}

We remark that a family being $\mathbb{Q}$-Fano is a more restrictive condition than being flat with every fiber $\mathbb{Q}$-Fano, because we also need that the canonical divisor of the global family is $\mathbb{Q}$-Cartier. 
Given a polarized generic $\mathbb{Q}$-Fano family $(\mX,\mL)\rightarrow C$, there exists a maximal open set $C^*\subset C$ parametrizing non-degenerate fibers. Conversely, given a $\mathbb{Q}$-Fano family $\mX^*$ over $C^*$ and $C^*\subset C$ where $C$ is a smooth curve,  using the properness of the Hilbert scheme, we easily see $\mX^*$ can be extended to be a generic $\mathbb{Q}$-Fano family over $C$. 

Second, we want our total family to minimize the following invariant which is motivated by the intersection number interpretation (See \cite{Wang, Oda} or Section \ref{int}) of the original DF invariant defined for a test configuration (see Definition \ref{defineTC} and Definition \ref{defineDFtc}). We refer to \cite{Fut, FS90, DT, Tian97, Dol, PauTi, PauTi2} and Remark \ref{history} in Section \ref{ss-KE} for background and history about this important invariant.

\begin{defn}[Donaldson-Futaki invariant]\label{defint} Let $\mL$ be a relative nef
$\mathbb{R}$-line bundle on $\mX$ over  a proper smooth curve $C$. We assume that there is  a non-empty open set $C^*\subset C$, such that  for any $t\in C^*$, $ \mX_t $ is klt and $\mL|_{\mX^*}\sim_{\mathbb{R}, C^*}-rK_{\mX^*}$ which is ample over $C^*$.  We
define the Donaldson-Futaki invariant (or DF invariant) to be
\begin{equation}\label{intform}
\Fut(\mX/C,\mL)=\frac{1}{2(n+1)(-K_{\mX_t})^n
}\left((n+1)K_{\mX/C}\cdot(\frac{1}{r}\mL)^n+(\frac{1}{r}\mL)^{n+1}\right).
\end{equation}
\end{defn}
 
This is proportional to the CM degree which is the degree of the CM line bundle (see e.g. \cite{PauTi2}). It is a very important invariant for studying family of Fano varieties whose positivity has not been fully understood. We refer Section \ref{ss-KE} for more background, especially its relation with the existence of K\"{a}hler-Einstein metric. 

Now we can state our main theorem, which roughly says certain `nice' degeneration exists. 

 \begin{thm}\label{main}
Assume that $(\mX, \mL)\rightarrow C$ is a polarized generic
$\mathbb{Q}$-Fano family over a smooth proper curve $C$. Let
$C^*\subset C$ parametrize non-degenerate fibers. We can construct 
a finite morphism $\phi: C'\rightarrow C$ from a non-singular curve
$C'$,  a $\mathbb{Q}$-Fano family $(\mX^{{\rm s}},\mL^{{\rm s}}:=-K_{\mX^{{\rm s}}})\to C'$
and a birational map $ \mX^{{\rm s}}\dasharrow
\mX\times_C C'$ which induces an isomorphism between the
restrictions
$$(\mX^{{\rm s}},\mL^{{\rm s}})|_{\phi^{-1} ( C^*)} \cong (\mX, \mL)|_{C^*}\times_C C' .$$
such that
\[
\Fut(\mX^{{\rm s}}/C',-K_{\mX^{{\rm s}}})\le \deg (\phi)\cdot \Fut(\mX/C,\mL)
\]
and the equality holds for the construction only if $(\mX,\mL)\to C$ is a $\mathbb{Q}$-Fano family.

\end{thm}



\subsection{Outline of the proof}

In this subsection, we explain our strategy. 
The main idea of showing Theorem \ref{main} is to modify a given a polarized generic $\mathbb{Q}$-Fano family, and then to use the intersection formula to analyze the
variation of Donaldson-Futaki invariants under modifications of the
test configurations. 

When the polarized generic $\mathbb{Q}$-Fano family arises from a test configuration, the authors in \cite{RT} and \cite{ALV} have also studied how to
simplify a given test configuration. In particular, by using the
(equivariant) semistable reduction theorem from \cite{KKMS}, it was
shown (cf. \cite{ALV}) that any K-unstable polarized variety $(X,L)$
can be destablized by a test configuration whose central fibre
$\mX^{{\rm tc}}_0$ is (reduced) simple normal crossing. In our paper, we will apply the machinery of minimal model program with scaling to modify this semistable test configuraton.
As far as we know, the idea of applying the modern birational geometry to study K-stability algebraically is first contained in Odaka's work  (see \cite{Oda1}). Here
we carry out a more refined study of the change of Donaldson-Futaki invariants   under various
birational modifications coming from MMP when $X$ is Fano.

Our first observation is that the DF invariants of the
polarizations always {\it decrease} along the direction of the
canonical class of the test configuration. Of course, when we deform
the polarization along the canonical class for enough long time, we
may hit the boundary of ample cone. Then MMP allows us to change the model and
enable us to continue the deformation. So as long as it is
before the pseudo-effective threshold,  we still get
polarizations but on some new models. In fact in birational geometry, this
process is precisely the minimal model program with scaling which is
a central theme in many recent work, e.g., \cite{BCHM}.

On the other hand, we can show if we run an MMP from the semi-stable model with the  scaling of a suitably perturbed  polarization, at the pseudo-effective threshold point we contract all but one components and end with a special test configuration. Since by our observation, the Donaldson-Futaki invariants decrease along the  sequence of MMP with scaling, this indeed finishes the proof of the $K$-semistable case.

However, to prove the $K$-stable case, we have to eliminate  the perturbation term
 appearing when we perturb the polarization.
This is more involved and therefore we have to take a more
delicate process.  In fact we have to  divide the modification process into three steps. 
\vspace{5mm}

\noindent{\bf Step 1.}
We first apply semi-stable reduction and run a relative MMP of this semi-stable family over $\mX$. Then we achieve a model $\mX^{{\rm lc}}$ which is the {\it log canonical modification}
of the base change of $(\mathcal{X},\mathcal{X}_0)$. We define the polarization $\mL^{{\rm lc}}$ on $\mX^{{\rm lc}}$ to be the sum of the pull back of the original polarization and a small positive multiple of $K_{\mX^{{\rm lc}}}$, i.e.,  the new polarization on $K_{\mX^{{\rm lc}}}$ is obtained by deforming the original one along the direction of the canonical class. Thus by our observation, we can check that the DF-invariants decrease.
We note that the idea of running the MMP by passing through
the log canonical modification is inspired by Odaka's work (see
\cite{Oda}). But here we only need to compute the derivative of the DF invariants. In Odaka's work \cite{Oda}, as he was studying the K-stability problem for general polarized varieties, a subtler estimate involving in terms of higher order was needed.

\begin{thm}\label{t-step1}
Let $(\mX,\mL)\to C$ be a polarized generic $\mathbb{Q}$-Fano family over a proper smooth curve, with $C^*\subset C$ parametrizing the non-degenerate fibers. Then we can construct  a finite morphism $\phi:C'\to C$ and a polarized generic $\mathbb{Q}$-Fano family $(\mX^{{\rm lc}},\mL^{{\rm lc}})\to {C}'$
 with the
following properties:
\begin{enumerate}
\item There is a birational morphism $\mX^{{\rm lc}}\rightarrow \mX\times_CC'$,
 which is isomorphic over $\phi^{-1}(C^*)$.
\item For every $t\in C'$,  $(\mX^{{\rm lc}},\mX^{{\rm lc}}_t)$ is log canonical.
\item There is an inequality
$$\Fut(\mX^{{\rm lc}}/C',\mL^{{\rm lc}})\le \deg (\phi) \cdot \Fut(\mX/C,\mL).$$
Moreover, the equality holds for our construction if and only if $(\mX,\mX_t)$ is log canonical for
every $t\in C$, in which case $\mX^{{\rm lc}}$ is isomorphic to $
\mX\times_{C}C'$.
\end{enumerate}
\end{thm}

\vspace{5mm} \noindent{\bf Step 2.} Next, we will run the minimal
model program with scaling. 
By letting $l>0$ be a sufficiently large integer, we can assume that $\mathcal{H}^{{\rm lc}}=\mL^{{\rm lc}}-(l+1)^{-1}(K_{\mX}^{{\rm lc}}+\mL^{{\rm lc}})$ is ample.
Thus we can run $K_{\mX^{{\rm lc}}}$-MMP with scaling of
$\mathcal{H}^{{\rm lc}}$ over $C$ starting from the polarization $K_{\mX^{{\rm lc}}}+(l+1){\mathcal{H}^{{\rm lc}}}=l\mL^{{\rm lc}}$. It follows from \cite{BCHM} that the
sequence of $K_{\mathcal{X}^{{\rm lc}}}$-MMP with scaling of
$\mathcal{H}^{{\rm lc}}$ over $C$ will decrease the scaling factor until the pseudo-effective
threshold is reached. We denote its anti-canonical model by
$\mX^{{\rm ac}}$. Since this is again deforming the polarization along the
direction of the canonical class, we can also check the DF
invariants are continuously decreasing when we scale down
the scaling factor. So we have the following theorem.

\begin{thm}\label{t-step2}
Let $(\mX^{{\rm lc}},\mL^{{\rm lc}})\to C$ be a polarized generic $\mathbb{Q}$-Fano family over a proper smooth curve, with $C^*\subset C$ parametrizing the non-degenerate fibers. We assume that $(\mX^{{\rm lc}},\mX^{{\rm lc}}_t )$ is log canonical for any $t\in C$, then we can construct a polarized generic $\mathbb{Q}$-Fano family $(\mX^{{\rm ac}},\mL^{{\rm ac}})\to C$ which is isomorphic to  $(\mX^{{\rm lc}},\mL^{{\rm lc}})$ over $C^*$, such that $-K_{\mX^{{\rm ac}}}\sim_{\mathbb{Q},C}\mL^{{\rm ac}}$, $(\mX^{{\rm ac}},\mX^{{\rm ac}}_t )$ is log canonical for any $t\in C$ and
$$ \Fut(\mX^{{\rm ac}}/C,\mL^{{\rm ac}})\le \Fut(\mX^{{\rm lc}}/C,\mL^{{\rm lc}}),$$
with the equality holding for our construction if and only if $(\mX^{{\rm ac}},\mL^{{\rm ac}})\cong (\mX^{{\rm lc}},\mL^{{\rm lc}})$.
\end{thm}

\vspace{5mm}
\noindent {\bf Step 3.}
At this point,  MMP of $\mX^{{\rm ac}}$ with scaling of $\mL^{{\rm ac}}\sim_{\mathbb{Q},C}-K_{\mX^{{\rm ac}}}$ will not produce any new model. 
Instead of running MMP, now we apply the technique of extending $\mathbb{Q}$-Fano 
varieties. So after possibly a base change we obtain a $\mathbb{Q}$-Fano family  $\mX^{{\rm s}}$ which is isomorphic over the base which parametrizes nondegenerate fibers. 
Furthermore, from our construction of $\mathbb{Q}$-Fano extension, we know that $\mX^{{\rm s}}_t$ is a divisor whose discrepancy with respect to $\mX^{{\rm ac}}$ is 0.

Thus we can extract $\mX^{\rm s}_t$ on $\mX^{{\rm ac}}$ to get a model ${\mX}'$ such that $-K_{{\mX}'}$ is relatively big and nef, and ${\mX}'\dasharrow \mX^{{\rm s}}$ is a birational contraction.
Since over those models whose anti-canonical class is proportional to the polarization over the base, the DF invariant is just a linear function on  the volume of the anti-canonical class whose linear term has a negative coefficient. Thus  we easily see it  decreases from ${\mX}'$ to $\mX^{{\rm s}}$. 

\begin{thm}\label{t-step3}
Let $(\mX^{{\rm ac}},\mL^{{\rm ac}})\to C$ be a polarized generic
$\mathbb{Q}$-Fano family over a proper smooth curve, where $C^{*}\subset C$ parametrizes
non-degenerate fibers. We assume that $(\mX^{{\rm ac}},\mX^{{\rm ac}}_t)$ is log canonical for any $t\in C$ and $-
K_{\mX^{{\rm ac}}}\sim_{\mathbb{Q},C}\mL^{{\rm ac}}$. Then after a finite base change
$\phi:C'\to C$, we can construct a $\mathbb{Q}$-Fano family, $(\mX^{{\rm s}},
-K_{\mX^{{\rm s}}})\to C'$ which is isomorphic to $(\mX^{{\rm ac}},\mL^{{\rm ac}})\times_C C'$ over $\phi^{-1}(C^*)$, such that
$$ \Fut(\mX^{{\rm s}}/C',-K_{\mX^{{\rm s}}})\le
\deg(\phi) \cdot \Fut(\mX^{{\rm ac}}/C,-K_{\mX^{{\rm ac}}}),$$ and the equality holds for our construction if and only if
$(\mX^{{\rm s}},-K_{\mX^{{\rm s}}})$ is a $\mathbb{Q}$-Fano family.
\end{thm}

Finally, we outline the organization of Part \ref{part1}. In Section
\ref{birational}, we review the results which we need from birational
geometry and MMP. At the end we prove Theorem \ref{t-dfano} which is
a weaker form of Theorem \ref{main} and will be needed for the proof of the general case. 
In Section \ref{lcm}, \ref{bfpsef} and
\ref{atpsef}, we prove Theorem \ref{t-step1}, \ref{t-step2} and
\ref{t-step3} respectively. At Section \ref{PET}, we put them
together to obtain Theorem \ref{main} the main theorem.

\section{MMP and birational models}\label{birational}

In this section, we give a few definitions as well as briefly introduce some preliminary results, especially the results in the minimal model program. Then we prove Theorem \ref{t-dfano}, which is can be considered as a weaker version of Theorem \ref{main}. Later this form is needed for the purpose of calculating the variance of Donaldson-Futaki invariants and thus to address the main Theorem.

\subsection{Notation and Conventions}\label{notation}

We follow \cite{KM} for the standard terminologies of birational
geometry.  In particular, see \cite[0.4]{KM} for the definitions of
general concepts and \cite[2.34, 2.37]{KM} for the definitions of
{\it klt, lc} and {\it dlt} singularities.

A smooth variety $Y$ which is flat over a smooth curve $C$ is called {\it a semi-stable family} if for any $t\in C$, $(Y,Y_t)$ is simple normal crossing.

Assume that $X$ is
proper over $S$. For any $\mathbb{Q}$-divisor $D$ on $X$, we denote $\oplus_m
f_*(\mathcal{O}_X(\lfloor m D \rfloor) )$ by $R(X/S,D)$. A $\mathbb{Q}$-Cartier $\mathbb{Q}$-divisor $L$ on
$X$ is called {\it relatively semi-ample} if for sufficiently divisible $m>0$,
$f^*f_*\mathcal{O}_X(mL)\to \mathcal{O}_X(mL)$ is surjective. 

Assume that $f:(X,\Delta)\to S$ is a log canonical pair projective over $S$, where
$K_X+\Delta$ is big and semi-ample over $S$, then we know  $ R(X/S,K_X+\Delta)$ is a finitely generated $\mathcal{O}_S$-algebra, and we define
$$Y={\rm Proj}\ R(X/S,K_X+\Delta) $$
{\it the relative log canonical model} of $(X,\Delta)$ over $S$. 
We say $X^{\rm m}$ is {\it a good minimal model} of $(X,\Delta)$ over $S$ if $h:X \dasharrow X^{\rm m}$ is a minimal model of $(X,\Delta)$ over $S$ and $K_{X^{\rm m}}+h_*\Delta$ is relatively semi-ample.

Let $(X,\Delta)$ be a normal pair,
i.e., $X$ is a normal variety and $\Delta=\sum a_i\Delta_i$ is a $\mathbb{Q}$-divisor
with distinct prime divisors $\Delta_i$ and rational numbers $a_i$.
Assume $0\le a_i \le 1$.
We say that a birational projective morphism $f: Y\to (X,\Delta)$  gives {\it the  log canonical modification} of $(X,\Delta)$  if with  the divisor $\Delta_Y=f^{-1}_*(\Delta)+E^{{\rm lc}}$ on $Y$, the pair $(Y,\Delta_Y)$ satisfies
\begin{enumerate}
\item[(1)] $(Y,\Delta_Y)$ is log canonical, and
\item[(2)] $K_Y+\Delta_Y$ is ample over $X$.
\end{enumerate}
Here $E^{{\rm lc}}$ denotes the sum of $f$-exceptional prime divisors with coefficients $1$.

Let $X\to Y$ be a dominant morphism between normal varieties. A prime divisor  $E$ of $X$ is called $horizontal$ if $E$ dominates $Y$, otherwise it is called $vertical$. Given any divisor $E$,  we can uniquely write $E=E^v+E^h$ where the horizontal part $E^h$ consists of all the horizontal components and  the vertical part $E^v$ consists of all the vertical components. 

For a model $\bullet$ (e.g., $\mX$) over $C$ and $0\in C$ a point, we denote by $\bullet_0$ its fiber over $0$.

\subsection{MMP with scaling}\label{ss-MMP}
In this subsection, we briefly introduce the concept of MMP with scaling and summarize the results which we will need later. For more details, see e.g. \cite{BCHM}.

Let $(X,\Delta)/S$ be a klt pair, which is projective over $S$. Let $H$ be an ample class on $X$.
Let $\lambda\ge 0$ be a positive number such that $K_X+\Delta+\lambda H$ is nef over $S$. For instance, we can take $\lambda \gg 0$.
Denote $(X^0,\Delta^0):=(X,\Delta)$ and $\lambda_0=\lambda$.

Suppose we have defined a klt pair $(X^{i}, \Delta^i)$ which is
projective over $S$, a $\mathbb{Q}$-divisor $H^i$ on $X^i$,  and a
positive number $\lambda_i$ such that
$K_{X^{i}}+\Delta^i+\lambda_iH^i$ is nef over $S$.  We first define
\begin{equation}\label{jumplam}
\lambda_{i+1}=\min\{\lambda | K_{X^{i}}+\Delta^i+\lambda H^{i} \mbox{ is nef over } S \}
\end{equation}
If $\lambda_{i+1}=0$, then we stop. Otherwise, suppose
we can choose a $(K_{X^i}+\Delta^i)$-negative extremal ray $R\subset {\rm NE}(X^i/S)$, with
$R\cdot  ( K_{X^{i}}+\Delta^i+\lambda_{i+1}H^{i})=0$. Assume that the
extremal contraction induced by $R$ is birational. Then we let
$X^{i+1}$ be the variety obtained by the $(K_{X^{i}}+\Delta^i)$-flip or
divisorial contraction over $S$ with respect to the curve class $R$
(cf. \cite[Section 3.7]{KM}), and $\Delta^{i+1}$ (resp. $H^{i+1}$)
the push-forward of $\Delta^i$ (resp. $H^i$) to $X^{i+1}$. It is
obvious that $\lambda_0\ge\lambda_1\ge \cdots \ge
\lambda_i\ge\cdots$. We call this sequence
$$(X^0,\Delta^0)\dashrightarrow(X^1,\Delta^1)\dashrightarrow\cdots \dashrightarrow(X^i,\Delta^i)\dashrightarrow\cdots,$$
 {\it a sequence of $(K_X+\Delta)$-MMP with scaling of $H$}, $\lambda$ {\it the scaling factor} and $H$  {\it the scaling divisor class} (or {\it the scaling divisor} for abbreviation). 

 From the definition, we know that for any $t\in[\lambda_{i+1},\lambda_{i}]$, $K_{X^i}+\Delta^i+tH^i$ is nef over $S$. Moreover, if $t\in[\lambda_{i+1},\lambda_{i})$, then $X^i$ is a relatively minimal model of  $(X, \Delta+tH)$ over $S$ because for $j\le i$, the step $X^{j-1}\dasharrow X^{j}$ of $(K_X+\Delta)$-MMP is also a step for $(K_{X}+\Delta+tH)$-MMP.

We need the following results.
\begin{thm}\label{t-mmp}With the above notations. 
Let $(X,\Delta)$ be a klt pair which is projective over $S$. There exists a finite number $i$ such that
\begin{enumerate}
\item if $\Delta$ is big and $K_X+\Delta$ is pseudo-effective  over $S$, then after  $i$ steps, $\lambda_i=0$ and $K_{X^i}+\Delta^i$ is  semi-ample over $S$;
\item if $ K_X+\Delta$ is not pseudo-effective, then after $i$ steps, $\lambda_i>0$ is equal to the pseudo-effective threshold (see \cite[1.1.6]{BCHM} for the definition) of $K_X+\Delta$ with respect to $H$, which is a rational number. Furthermore,   if we let $i$ be the smallest index such that $\lambda_i$ is equal to the pseudo-effective threshold, then $X^{i-1}$ is a good minimal model of $(X,\Delta+\lambda_iH)$ over $C$.
\end{enumerate}
\end{thm}
\begin{proof} In (2), running $(K_X+\Delta)$-MMP with scaling of $H$ is the same as running $(K_X+\Delta+(\lambda_i-\epsilon)H)$-MMP with scaling of $H$ for $0<\epsilon<\lambda_i$, therefore these statements follow from \cite[1.1.7, 1.3.3 and 1.4.2]{BCHM}.
\end{proof}

\begin{prop}\label{p-lai}
Let $(Y,\Delta_Y)$ be a klt pair projective over a smooth curve $C$ with a relative ample class $L$. We assume that
 we can write $K_{Y}+\Delta_Y+L\sim_{\mathbb{Q},C}E=E^{h}+E^{v}\ge 0$ such that  the horizontal part $E^{h}$ is exceptional for a birational morphism $Y\rightarrow X$, and the vertical part $E^{v}$ can be written as $\sum a_iE^v_i$ where $a_i>0$, and ${\rm Supp}(\sum E^{v}_i)$ does not contain any fiber.
 Then we have the following:
 \begin{enumerate}
 \item The $(K_Y+\Delta_Y+L)$-MMP with scaling of $L$ will yield a model $Y^i$ such that $K_{Y^i}+\Delta_{Y^i}+L^i\sim_{\mathbb{Q}}0$ where $\Delta_{Y^i}$ and $L^i$ are the push forward of $\Delta_Y$ and $L$ on $Y^i$.  
 \item The divisors contracted by $Y\dasharrow Y^i$ are precisely ${\rm Supp}(E)$, and $L^i$  on $Y^i$ is relatively  big and nef over $C$.
\end{enumerate}
\end{prop}
\begin{proof} From the assumption, we know that the pseudo-effective threshold of $(Y,\Delta_Y)/C$ with respect to $L$ is 1. Then by Theorem \ref{t-mmp}(2), this sequence of MMP yields a good minimal model $Y^i$ of $K_{Y}+\Delta_Y+L$ over $C$. Since $K_{Y^i}+\Delta_{Y^i}+L^i$ is semi-ample, the map $Y\dasharrow Y^i$ contracts precisely the divisorial component in ${\bf B}(K_Y+\Delta_Y+L/C)$ which is ${\rm Supp}(E)$. In fact, it is easy to see this for the components in $E^h$ since they are exceptional for a birational morphism. For $E^v$, by our assumption, it is of the insufficient fiber type (cf. \cite[2.9]{Lai}), so by \cite[2.10]{Lai} it is contained in
$${\bf B}_{-}(K_Y+\Delta_Y+L/C)  \subset{\bf B}(K_Y+\Delta_Y+L/C).$$

From the definition of MMP with scaling, we see that for any $t\in[\lambda_{i+1},\lambda_{i}]$, $K_{Y^i}+\Delta_{Y^i}+tL^i$ is nef. Since $K_{Y^i}+\Delta_{Y^i}+L^i\sim_{\mathbb{Q},C} 0$ and by our assumption $\lambda_{i}>\lambda_{i+1}=1$,  then $L^i$ is nef over $C$.
\end{proof}

\subsection{Log canonical modification and $\mathbb{Q}$-Fano
extension}\label{sub-c}
Let $f^* : \mX^*\to C^*$ be a flat projective
morphism, $\mX^*$ a klt variety, $C^*$ the germ of a smooth curve.
Let $ C$ be a smooth curve such that $C^*= C \setminus \{0\}$. Let
${\mX}$ be a normal compacitifcation of $\mX^*$ which is projective
over $C$ such that $\mX^*=\mX\times_C C^*$.  We first show a general result of the
existence of the log canonical modification for the variety fibered
over a curve. In fact, the log canonical modification is known  to exist under more general assumptions (see e.g., \cite{OX}). Here we just give a proof of the case that we need for the reader's convenience.
\begin{prop}\label{lc}
With the above notations. Assume that $(\mX,\mX_0)$ admits a log resolution
$\mY$, such that $\mY_0$ is reduced simple normal crossing. Then the
log canonical modification $\mX^{{\rm lc}}\to (\mX,\mX_0)$ exists and
satisfies $(\mX^{{\rm lc}},\mX^{{\rm lc}}_0)$ is log canonical.
\end{prop}
\begin{proof} Let $\pi:\mathcal{Y}\to (\mX,\mX_0)$ be a log resolution.
If we write $B$ to be the reduced divisor ${\rm Ex}(\pi)$, it is
well-known that it suffices to show that
$(\mY,B+\pi_*^{-1}\mX_0)$ has a relative log canonical model over
$\mX$ (see \cite[Lemma 2.2]{OX}). Write
$$\pi^*K_{\mathcal{X^*}}+F^*=K_{\mathcal{Y^*}}+E^*,$$ where $F^*$, $E^*$
are effective and without common components. Let $E$ be the closure
of $E^*$ in $\mathcal{Y}$. Now we consider the pair $(\mY,E+\delta
G)$ where $G$ is the sum of the $\pi$-exceptional divisors whose centers are
over $C^*$ and $0<\delta\ll 1$ such that $(\mY,E+\delta G)$ is klt.
Then it follows from \cite{BCHM} that $R(\mY/\mX, K_{\mY}+E+\delta
G)$ is a finitely generated $\mathcal{O}_{\mX}$-algebra. By taking
${\rm Proj}$, we obtain  a model $\phi:\mathcal{Y}\dasharrow \mX^{{\rm lc}}$ over $\mX$.
The model $\mX^{{\rm lc}}$ is isomorphic to ${\rm Proj}R(\mathcal{Y}/\mX,K_{\mathcal{Y}}+E+\delta G+\mY_0)$  since $\mY_0$ is the pull back of a divisor from $\mX$.

Because
$D=B+\pi_*^{-1}\mX_0-E-\delta G-\mY_0\ge 0$ is an effective divisor, we know that
$$K_{\mY}+B+\pi_*^{-1}\mX_0-\phi^*\phi_*(K_{\mY}+E+\delta G+\mY_0)\ge K_{\mY}+B+\pi_*^{-1}\mX_0-(K_{\mY}+E+\delta G+\mY_0)\ge 0.$$
Since $\phi$ contracts  $G$ which is the same as ${\rm Supp}(D)$, we easily see this implies that  there is an
isomorphism
$$R(\mathcal{Y}/\mX, K_{\mY}+\pi_*^{-1}\mX_0+B)\cong  R(\mathcal{Y}/\mX,K_{\mathcal{Y}}+E+\delta G+\mY_0). $$
Hence we see $\mathcal{X}^{{\rm lc}}$ is indeed the log canonical modification of $(\mX,\mX_0)$ and $(\mX^{{\rm lc}},\mX^{{\rm lc}}_0)$ is log canonical as $\phi_*(E+\delta G+\mY_0)=\mX_0^{{\rm lc}}$.
\end{proof}

Next we study degenerations of Fano varieties.

\begin{exmp}[Degenerations of cubic surfaces]\label{ex-deg}
Let us consider a family of cubic surfaces:  
$$\mX=(t f_3(x,y,z,w)+sxyz=0)
\qquad \subset \qquad  \mathbb{P}(x,y,z,w)\times \mathbb{P}(s,t), 
$$
where $f_3$ (resp. $g_3$) is a general degree 3 homogeneous polynomial of $x,y,z$ and $w$ (resp. $x,y$ and $z$). Projecting to the second factor,  $\mX$ 
are families of cubic surfaces over $\mathbb{A}^1$ whose general fibers are smooth.

 Now we modify $\mX$ in the following way: First we blow up the point
 $$(0,0,0,1)\in  \mX_0=\sum^3_{i=1} E_i=(xyz=0)\subset \mathbb{P}^3$$
 to get $\mX'$ and we denote the exceptional divisor by $S_0\cong \mathbb{P}^2$. The fiber $\mX'_0$ has multiplicity 3 along $S_0$. Each birational transform of $E_i$  is isomorphic to the $\mathbb{P}^1$-bundle $\mathbb{P}(\mathcal{O}\oplus\mathcal{O}(1))$ over $\mathbb{P}^1$.
 
 Next we take a degree 3 base change $\mathbb{P}[s_1,t_1]\to \mathbb{P}[s,t]$ which sends $t$ (resp. $s$) to $t_1^3$ (resp. $s_1^3$). Let $\tilde{\mX}$ be the normalization of $\mX'\times_{\mathbb{P}^1}\mathbb{P}^1$. The pre-image $ S_1$ of $S_0$ in $\tilde{\mX}$ is the degree 3 cover branched over the intersection of $E$ and the birational transform of $\sum^3_{i=1}E_i$, which is isomorphic to $(xyz=0)\subset  \mathbb{P}(x,y,z)\cong S_0$.  Hence $S_1$ is a cubic surface with three $A_2$ singularities. We can contract the preimage of the birational transform $T$ of $\sum^3_{i=1}E_i$ and get a model $\mX^{{\rm s}}$. It is easy to see that $\mX^{{\rm s}}_0\cong S_1$.
 
\end{exmp}


This construction can be indeed generalized to arbitrary polarized generic $\mathbb{Q}$-Fano families in the following sense.

\begin{thm}[$\mathbb{Q}$-Fano extension]\label{t-dfano}Let $f: \mX \to C$ be a projective
morphism over a smooth curve $C$. Assume that $\mX^*=\mX\times_C C^*$ is a klt variety, where $C^*= C \setminus \{0\}$.
We assume that there exists an ample $\mathbb{Q}$-divisor $\mL$ on
$\mX$ such that $\mL|_{\mX^*}\sim_{\mathbb{Q},C^*}-K_{\mX^*}$ and 
$\mX^*_t$ is a $\mathbb{Q}$-Fano variety for any $t\in C^*$.
\begin{enumerate}
\item  There is a finite morphism $\phi: C'\rightarrow C$ and  a klt variety $\mX^{{\rm s}}$ projective over $C'$ such that  the restriction of $\mX^{{\rm s}}$ to the preimage $\phi^{-1}(C^*)$ is isomorphic to $\mX^*\times_C C'$ and all fibers $\mX_t$ is $\mathbb{Q}$-Fano variety. In particular, $\mX_t$ is normal.

\item Moreover, if we assume that $f^{{\rm ac}}:\mX=\mX^{{\rm ac}}\to C$ is a normal compactification such that $\mL^{{\rm ac}}=-K_{\mX^{{\rm ac}}}$ is anti-ample and for any $t\in C$, $(\mX^{{\rm ac}},\mX^{{\rm ac}}_t)$ is log canonical. We can indeed get a $\mathbb{Q}$-Fano extension $\mX^{{\rm s}}$ in (1) such that the divisor $\mX^{{\rm s}}_t$ as a valuation over $\tilde{\mX}^{{\rm ac}}:=\mX\times_C C'$ has the discrepancy $a(\mX^{{\rm s}}_t; \tilde {\mX}^{{\rm ac}})=0$.
\end{enumerate}
\end{thm}

\begin{proof}Let $\phi: C'\to C$ be a finite base change such that $ \mX\times_{C} C' $ yields a semi-stable log resolution $\pi:\mY \to \mX\times_{C} C'$. We can assume that $\mY$ yields an exceptional divisor $A$ which is $\pi$-ample.  Let $\pi^*$ be the restriction of $\pi$ over ${C}^*$. By abuse of notation, we will identify $C=C'$ and $\mX^*=\mX\times_{C}C'^*$.

Write $\pi^*(K_{\mX^*})+F^*=K_{\mY^*}+\Gamma^*$, where $\Gamma^*,F^*>0$ are
effective divisors without common components, by our assumption the
coefficients of $\Gamma^*$ are less than 1. Let $\Gamma,F$ be the closures of
$\Gamma^*$ and $F^*$. Let $\epsilon$ be a sufficiently small positive
number such that $\mL_{\mY}:=\pi^*\mL+\epsilon A$  is ample.  We
write $A=A_1+A_2$, such that $A_2$ precisely consists of the
vertical components which are over $0 $. By perturbing $\mL_\mY$ and
reordering the components, we can assume that if we write
$$K_{\mY}+\mL_{\mY}+\Gamma\sim_{\mathbb{Q}, C}(\epsilon A_1 +F)+(\epsilon A_2 +B)$$
with $B$ supported on the fibers $\mY_{0}=\sum E_{j}$ and
$B+\epsilon A_2= \sum^{k}_{j=1} a_{j}E_{j},$  then $a_{j} >a_{1}$ if
$j>1$, where $E_j$ ($1\le j\le k$) are all the components of $\mY_0$.

 Let $G$ be the sum of the prime divisors which are $\pi$-exceptional divisors whose centers are in $\mX^*$.
 By choosing $\epsilon \ll \delta\ll 1$, we assume that  $\delta G+\epsilon A_1\ge 0$ and its support is equal to $G$. We run $(K_{\mY}+\mL_{\mY}+\Gamma+\delta G)$-MMP  over ${\mX}$ with  scaling of $\mL_{\mY}$, by adding multiples of fibers, we can also assume that $a_1=0$ in the above formula.
Because
$$(K_{\mY}+\mL_{\mY}+\Gamma+\delta G)\sim_{\mathbb{Q}, C} (\delta G+\epsilon  A_1 +F)+(\epsilon A_2+B-a_1\mY_0)\ge 0,$$
whose support restricting on $\mY^*$ contains all the exceptional
divisors for $\mY^* \to \mX^*$, we can apply Proposition \ref{p-lai}
and conclude this sequence of MMP terminates with a model $\mX^{{\rm m}}$
satisfying that $K_{\mX^{{\rm m}}}+\mL^{{\rm m}}\sim_{\mathbb{Q},C} 0$, the only
remaining component over $0$ is the birational transform of $E_1$
and  $\mL^{{\rm m}}$ is big and nef. Thus  we can define  $\mX^{{\rm s}}$ to be
${\rm Proj}R(\mX^{{\rm m}}/C,-K_{\mX^{{\rm m}}})$. Over $C^*$, since
$\mY^*\dasharrow \mX^{{\rm m}*}$ contracts the same components as the ones
of $\mY^*\to \mX^*$, thus $\mX^{{\rm m}*}(:=\mX^{\rm m}\times_C C^*)\dasharrow \mX^*$ is isomorphic in
codimension 1. Hence we see that
$$\mX^*={\rm Proj}R(\mX^*/C^*, \mL|_{\mX^*})\cong {\rm Proj}R(\mX^{{\rm m}*}/C^*, \mL^{{\rm m}}|_{\mX^{{\rm m}*}})=\mX^{{\rm s}*}.$$

Representing $\mL_{\mY}$ by a general $\mathbb{Q}$-divisor, we can
assume that $({\mY}, \Gamma+\delta G+\mY_{0}+\mL_{\mY})$ is dlt, The MMP
sequence is also  a sequence of $(K_{\mY}+\Gamma+\delta
G+\mY_{0}+\mL_{\mY})$-MMP, thus  $(\mX^{\rm m},\mX^{\rm m}_{0}+\mL^{\rm m})$ is dlt. This implies that $(\mX^{\rm m}, \mX^{\rm m}_{0})$ is
dlt since $\mX^{\rm m}$ is $\mathbb{Q}$-factorial. As $\mX^{\rm m}_0$ is irreducible, this indeed says
$(\mX^{{\rm m}},\mX^{{\rm m}}_{0})$ is plt and so $(\mX^{{\rm s}},\mX^{{\rm s}}_{0})$ is plt. By adjunction, we know $\mX^{{\rm s}}_0$ is klt. 
This finishes the proof of (1).

For (2), we apply the same line of argument.  We first choose  ${\mX}=\mX^{{\rm ac}}$. Then we know that we can write $K_{\mathcal{Y}}+\Gamma=\pi^*(K_{\mX^{{\rm ac}}})+F+B$, where $B$ is over $\{0\}$. Since $(\mX^{{\rm ac}}, \mX^{{\rm ac}}_0)$ is log canonical, $\mX^{{\rm ac}}$ is canonical along $\mX^{{\rm ac}}_0$. So $B\ge 0$,  whose support is the union of those divisors $E_{j}\subset \mathcal{Y}_0$ such that $a(E_{j}, \mX^{{\rm ac}})>0$. Now we have
$$K_{\mY}+\mL_{\mY}+\Gamma+\delta G\sim_{\mathbb{Q}, {\mX}} B+F+\epsilon A_1+\epsilon A_2+\delta G,$$
whose vertical part over $\{0\}$ is $B+\epsilon A_2$. Thus by choosing $0<\epsilon \ll 1$, we can assume that after
a small suitable perturbation, the divisor  $E_1$ having the  smallest coefficient
$a_1$ is not contained in ${\rm Supp}(B)$, i.e., it satisfies that $a(E_1, \mX^{{\rm ac}})=0$. Then from the proof of
(1), $\mX^{\rm m}_0$ will be the birational image of such $E_1$.
\end{proof}

As investigated in \cite{Kol07, HoXu}, the component $E_1$ constructed in this way has many interesting geometric properties. Since any $\mathbb{Q}$-Fano variety is rationally connected (see \cite{Zhang06}), the argument  which we just presented indeed gives a new proof of \cite[Theorem 3.1]{HoXu} which was originally obtained by applying Hacon-M$^c$Kernan's extension theorem as in \cite{HM07}.

\subsection{Zariski lemma }
Using this intersection formula, in the following work we need the higher
dimensional analogue of the Zariski's lemma for surfaces.
\begin{lem}\label{Zariski}
Let $\mX \to C$ be a projective dominant morphism from a
$n$-dimensional normal variety to a proper smooth curve. Let $E$ be
a $\mathbb{Q}$-divisor whose supports is contained in a fiber $\mX_0$. Let
$\mL_1,..,\mL_{n-2}$ be $n-2$ nef divisors on $\mX$. Then
$$E^2\cdot \mL_1\cdots \mL_{n-2}\le 0.$$
If all $\mL_i$'s are ample, then
 the equality holds if and only if $E =t\mX_{0}$ for some $t\in
\mathbb{Q}$.
\end{lem}
\begin{proof}When $n=2$, this is the well-known Zariski lemma (see e.g. \cite[III.8.2]{BHPV}). We note that since $E$ is not necessarily $\mathbb{Q}$-Cartier, here we have to use the intersection theory on normal surfaces developed by Mumford (see \cite{Mum}).
 For $n>2$, as a nef bundle is the limit a sequence of ample $\mathbb{Q}$-line bundles, we can assume all $\mL_i$ are ample $\mathbb{Q}$-line bundles. Replacing $\mL_i$ by a multiple, we assume $\mL_i$ to be very ample. Adding a suitable multiple of the fiber, we can assume that $E\ge 0$ and its support does not contain ${\rm Supp}(\mX_0)$. Therefore, we can cut $\mX$ by $n-2$ general sections in $|\mL_i|$ $(1\le i \le n-2)$ and reduce the question to the case when $n=2$.
\end{proof}


\section{Decreasing of DF invariant for the log canonical modification}\label{lcm}
 Let $(\mX,\mathcal{L})\rightarrow C$ be a
polarized generic $\mathbb{Q}$-Fano family.  It is easy to see all
our operations will be local over $C$, so to simplify the notation, without loss of generality
we will just denote one degenerate fiber to be $\mX_0$ and argue in
a neighborhood of it.

In this section, we aim to verify Theorem \ref{t-step1}.  In fact,
we will calculate the DF invariants on the log canonical
modification $\mX^{{\rm lc}}\to \mX$. We start from the line bundle
$\pi^{{\rm lc}*}\mL$ which is the pull back of the polarization on $\mX$,
whose DF invariant is equal to the original one. Since
$K_{\mX^{{\rm lc}}}$ is relative ample, if we deform $\pi^{{\rm lc}*}\mL$ along
the direction $ K_{\mX^{{\rm lc}}}$ sufficiently small amount then we get
an ample bundle on $\mX^{{\rm lc}}$. As this is a deformation along the
canonical class, we can show the DF invariants decrease
along this deformation.

\begin{prop}\label{P-DEC}
With the above notations. If $\mX^{{\rm lc}}$ is not isomorphic to $\mX$,
then we can choose a polarization $\mL^{{\rm lc}}$ on $\mX^{{\rm lc}}$ such that
$$\Fut(\mX^{{\rm lc}}/C,\mL^{{\rm lc}} )< \Fut(\mX/C,\mL).$$
\end{prop}

\begin{proof}
By definition, $K_{\mX^{{\rm lc}}}$  is $\pi^{{\rm lc}}$-ample.
We choose the relatively $\pi^{{\rm lc}}$-ample $\mathbb{Q}$-divisor
$$E=K_{\mX^{{\rm lc}}}+\pi^{{\rm lc}*}(\mL).$$ Then $E$ is
$\mathbb{Q}$-linearly equivalent to a divisor whose support is contained in
$\mathcal{X}^{{\rm lc}}_0$. Since for sufficiently small rational $\epsilon$,
$\mL^{{\rm lc}}_{t}=\pi^{{\rm lc}*}\mL+t E$ is ample for any $0<t<\epsilon$, we see that
$(\mX^{{\rm lc}},\mL^{{\rm lc}})\to C$ is also a polarized generic
$\mathbb{Q}$-Fano family.

Using the formula, we compute the derivative at $t_0\in (0,\epsilon)$:
$$\frac{d}{dt}\Fut(\mX^{{\rm lc}}/C ,\mathcal{L}_{t}^{{\rm lc}})|_{t_0}=n(n+1)C_0\cdot\left((\mathcal{L}_{t_0}^{{\rm lc}})^n\cdot E+K_{\mX^{{\rm lc}}}\cdot(\mL_{t_0}^{{\rm lc}})^{n-1}\cdot E\right)$$
$$=C_1 \cdot (\mL_{t_0}^{{\rm lc}})^{n-1}\cdot E\cdot \left(\mL_{t_0}^{{\rm lc}}+K_{\mX^{{\rm lc}}} \right)=C_1\cdot (\mL^{{\rm lc}}_{t_0})^{n-1}\cdot E^2,
$$
where $C_0$ and $C_1$ are positive numbers. By Lemma
\ref{Zariski}, the intersection $(\mL^{{\rm lc}}_{t_0})^{n-1}\cdot E^2 \le 0$
and it is zero if and only if $E=K_{\mX^{{\rm lc}}}+\pi^{{\rm lc}*}\mL$
is $\mathbb{Q}$-linearly equivalent to $a\mX_0^{{\rm lc}}$ for some $a$.
But this implies that $\mX^{{\rm lc}}\cong \mX$ since $K_{\mX^{{\rm lc}}}\sim_{\mathbb{Q},\mX}{E}$ is
$\pi^{{\rm lc}}$-ample.
\end{proof}

\begin{proof}[Proof of Theorem \ref{t-step1}]

First we can take the base change $\mX\times_C C'$ such that its
normalization $\tilde{\mX}$ admits a
semi-stable reduction $\mY$. In particular $\tilde{\mX}_0$ is
reduced. Let $\phi_{\mX}: \tilde{\mX}\rightarrow \mX$ be the natural
finite morphism and $\tilde{\mL}=\phi_{\mX}^*\mL$. We first note that
\begin{claim}
$$\deg(C'/C)\cdot\Fut(\mX/C,\mL)\ge
\Fut(\tilde{\mX}/C',\tilde{\mL}).$$
Furthermore, the equality holds
if and only if $\mX_0$ is reduced.
\end{claim}

Indeed, by the pull-back formula for the log differential, we have $K_{\tilde{\mX}}+\tilde{\mX}_0=f^*(K_{\mX}+{\rm
red}(\mX_0))$ and $K_{C'}+\{0'\}=\phi^*(K_C+\{0\})$.  So
\[
K_{\tilde{\mX}/C'}=f^*(K_{\mX/C}+({\rm
red}(\mX_0)-\mX_0))
\]
and the claim follows from the projection formula.

Now it follows from Proposition \ref{lc} that the log canonical
modification $\pi^{{\rm lc}}: \mX^{{\rm lc}}\to \tilde{\mathcal{X}}$ exists and satisfy
that $\pi^{{\rm lc}}$ is an isomorphism over $\phi^{-1}(C^*)$.

And then Proposition \ref{P-DEC} shows that
$$\Fut(\tilde{\mX}/C',\tilde{\mL})\ge
\Fut(\mX^{{\rm lc}}/C',\mL^{{\rm lc}}).$$

If $(\mX,\mX_0)$ is log canonical, then $\mX_0$ is reduced
and $(\tilde{\mX},\tilde{\mX}_0)$ is log canonical (cf.
\cite[5.20]{KM}), which implies $\mX^{{\rm lc}}\cong \tilde{\mX}$, therefore the equality holds.

Conversely, $\deg(\phi)\cdot\Fut(\mX/C,\mL)=
\Fut(\mX^{{\rm lc}}/C',\mL^{{\rm lc}})$ is equivalent to saying the above two
inequalities are indeed equalities. By Proposition 
\ref{P-DEC} and the above claim, this
holds only if $\mX_0$ is reduced and $\mX^{{\rm lc}}\cong \tilde{\mX}$ which implies $(\tilde{\mX},\tilde{\mX}_0)$
is log canonical. Since
$$\phi_{\mX}^*(K_{\mX}+\mX_0)=K_{\tilde{\mX}}+\tilde{\mX}_0,$$
 it follows that $({\mX}, {\mX}_0)$ is also log canonical (see \cite[5.20]{KM}).
\end{proof}


\section{MMP with scaling}\label{bfpsef}
In this section, we aim to prove Theorem \ref{t-step2}. We will
apply the idea that the Donaldson-Futaki invariants decrease if we deform the polarization along the direction
of the canonical class of the total family in `a long time' process. To keep the deformed line
bundle being a polarization, we have to do a sequence of surgeries
on the family. In algebraic geometry, this surgery is given by the MMP with scaling
(see \cite{BCHM} and Subsection \ref{ss-MMP}).

\subsection{Running MMP}

By taking $l>0$ to be a sufficiently large integer, we can make 
$\mathcal{H}^{\rm lc}=\mL^{\rm lc}-(l+1)^{-1}(\mL^{\rm lc}+K_{\mX}^{\rm lc})$ ample.
Let $\lambda_0=l+1$. We let $\mX^{0}=\mX^{{\rm lc}}$,
$\mL^0=\mL^{{\rm lc}}$ and $\mathcal{H}^0=\mathcal{H}^{{\rm lc}}$.  Then
$K_{\mX^0}+\lambda_0\mH^0=l \mL^{\rm lc}$ is relatively ample.


Given an exceptional divisor $E$, if its center dominates $C$ then
$a(E,\mX^{0})>-1$ because $\mX^*$ is klt; if its center is vertical
over $C$, then $a(E,\mX^{0})\ge 0$, since $(\mX^0,\mX^0_t)$ is log
canonical for any $t$ in $C$. In particular, $\mX^0$ is klt.  To
simplify the family, we run a sequence of $K_{\mX^0}$-MMP over $C$
with scaling of $\mH^0$ as in Subsection \ref{ss-MMP}. So we obtain a sequence of models
$$\mX^0\dashrightarrow \mX^1\dashrightarrow \cdots \dashrightarrow
\mX^k.$$ Recall that, as in Subsection \ref{ss-MMP}, we have a
sequence of critical value of scaling factors
\[
\lambda_{i+1}=\min\{\lambda\; |\; K_{\mX^{i}}+\lambda\mH^i \mbox{ is
nef over } C \}
\]
with $l+1=\lambda_0\ge \lambda_1\ge ... \ge
\lambda_k>\lambda_{k+1}=1$. Note that $\lambda_{k+1}=1$ is
the pseudo-effective threshold of $K_{\mX^0}$ with respect to
$\mH^0$ over $C$, since it is the pseudo-effective threshold for the
generic fiber. Any $\mX^i$ appearing in this sequence of
$K_{\mX^0}$-MMP with scaling of $\mH^0$ is  a relative weak log
canonical model of $(\mX^0,t\mH^0)$ for any $t\in [\lambda_i,\lambda_{i+1}]$ (see \cite[3.6.7]{BCHM} for the
definition of weak log canonical model).

For $\lambda>1$, we denote by
\begin{equation}\label{nvarp}
\mL_\lambda=\frac{1}{\lambda -1}(K_{\mX^0}+\lambda\mH^0).
\end{equation}
Let $\mL^{i}_\lambda$ (resp. $\mH^{i} $) be the push forward
of  $\mL^0$ (resp. $\mH^0$) to $\mX^i$.  As is clear from the
context, this should not be confused with the $i$-th power or
intersection product of $\mL_\lambda$ (resp. $\mH$). We note that by
definition $\mL_{l+1}^0=\mL^0.$

Given a $\lambda$, we will be interested in those $i$ which satisfy
that $\lambda_{i}\ge \lambda\ge \lambda_{i+1}$. Note that
\begin{equation}\label{ltoh}
K_{\mX^{i}}+\mL^{i}_\lambda=\frac{\lambda }{\lambda
-1}\left(K_{\mX^{i}}+\mH^i\right)
\end{equation}

\begin{lem}
$-K_{\mX^k}\sim_{\mathbb{Q},C}\mL^k_{\lambda_k}$ is big and
semi-ample over $C$.
\end{lem}
\begin{proof}
Since $\lambda_k>\lambda_{k+1}=1$, by \eqref{ltoh},
$$K_{\mX^k}+\mL^{k}_{\lambda_{k}}\sim_{\mathbb{Q}}\frac{\lambda_{k}}{\lambda_{k}-1}\left(K_{\mX^{k}}+\mH^{k}\right).$$
This line bundle is relatively nef over $C$ and its restriction to
the generic fiber is trivial, so it is $\mathbb{Q}$-linearly
equivalent to a linear sum of components of $\mX_0^k$. By its
nefness, we can apply Lemma \ref{Zariski} to get
$$K_{\mX^k}+\mL^k_{\lambda_k}\sim_{\mathbb{Q},
C} 0.$$ By \eqref{nvarp}, $\mL^k_{\lambda_k}$ is proportional to
$K_{\mX^k}+{\lambda_k}\mH^k$ which is big because
$\lambda_k>1$. From the relative base-point free theorem
(cf. Theorem 3.3 in \cite{KM}), it is semi-ample over $C$.
\end{proof}

By the above Lemma, we can define
$$\mX^{{\rm ac}}={\rm Proj} R(\mX^k/C, \mL^k_{\lambda_k})={\rm Proj} R(\mX^k/C, -K_{\mX^k/C}).$$
Since $(\mX^0,\mX_0^0)$ is log canonical and $\mX_0^0=(f \circ
\pi^{{\rm lc}})^*(\{ 0\})$, this is a also a sequence of
$(K_{\mX^0}+\mX_0^0)$-MMP and thus $(\mX^{k},\mX_0^{k})$ is log
canonical which implies that $(\mX^{{\rm ac}},\mX_0^{{\rm ac}})$ is log
canonical as well.

\subsection{Decreasing of DF-invariant}

For any $\lambda>1$, the restriction of
$K_{\mX^0}+\lambda\mathcal{H}^0$ over $C^*$ is relatively ample. So the MMP
with scaling does not change $\mX^{0}\times_C C^*$, i.e.,
$\mX^{0}\times_C C^*\cong \mX^{i}\times_C C^*$ for any $i\le k$.

Note that by the above lemma and projection formula,
\[
\Fut(\mX^k/C,\mL_{\lambda_k}^k)=\Fut(\mX^k/C,
-K_{\mX^k})=\Fut(\mX^{{\rm ac}}/C, -K_{\mX^{{\rm ac}}}).
\]
So Theorem \ref{t-step2} follows from the following Proposition.
\begin{prop}\label{P-DFMMP}With the notations as above, we have
$$
\Fut(\mX^0/C,\mL^0)\ge
\Fut(\mX^k/C,\mL^{k}_{\lambda_k})=\Fut(\mX^k/C, -K_{\mX^k}).
$$
The first equality holds if and only if $h:\mX^0\dashrightarrow
\mX^k$ is an isomorphism.
\end{prop}

To prove Proposition \ref{P-DFMMP}, we first study how DF invariants change when we run MMP with scaling and modify the polarization correspondingly. 
\subsubsection{Decreasing of DF on a fixed model}\label{Non-in}

Assume that $\mX_0^0=\sum_{j\in I}E_j$, where $E_j$'s are the prime
divisors. Since $(\mX^0,\mL^0_\lambda) \times_C C^*$ is isomorphic
to $(\mX^0\times_CC^*, -K_{\mX^0\times_CC^*})$, there exist
$a_j(\lambda)\in\mathbb{R}$ such that
$$
K_{\mX^0}+\mL_{\lambda}^0\sim_{\mathbb{R},C}\sum_{j\in
I}a_j(\lambda) E_j.
$$

On $\mX^i$, for any rational number $\lambda>1$ satisfying $\lambda_i\ge
\lambda \ge \lambda_{i+1}$, we know $\mL^{i}_\lambda$ is a
big and semi-ample. Let
$\mathcal{Z}_{\lambda}$ be the relative log canonical model of
$(\mX^0,\lambda \mH^0)$ over $C$. Then there is a morphism
$\pi_{\lambda}:\mX^i\to \mathcal{Z}_{\lambda}$ and a relatively ample
$\mathbb{Q}$-divisor $\mathcal{M}_{\lambda}$ on $\mathcal{Z}_{\lambda}$
whose pull back is
$$\frac{1}{\lambda -1}(K_{\mX^i}+\lambda \mH^i)=\mL^i_{\lambda}.$$

\begin{lem}\label{FUT}
If $\lambda_i\ge a > b\ge \lambda_{i+1}$ and $b>1$, then
$\Fut(\mX^{i},\mL_a^{i})\ge \Fut(\mX^{i},\mL_b^{i})$. The inequality is strict
 if  there is a rational number $\lambda\in [a, b]$, such that  the push
forward of $\sum_{j\in I}a_j(\lambda) E_j$ to
$\mathcal{Z}_{\lambda}$  is not a multiple of the pull back of
$0\in C$ on $\mathcal{Z}_{\lambda}$.
\end{lem}
\begin{proof}
We compute the derivative of the
DF invariants in a similar way as Propostion \ref{P-DEC}.
\begin{eqnarray*}
\frac{d}{d\lambda}\Fut(\mX^{i}/C,\mL_\lambda^{i})&=&C_0\left((\mL^{i}_\lambda)^{n-1}\cdot(\mL^{i}_\lambda)'\cdot(\mL^{i}_\lambda+K_{\mX^{i}})\right)\\
&=&-\frac{C_0}{\lambda(\lambda-1)}(\mL^{i}_\lambda)^{n-1}\cdot\left(K_{\mX^{i}}+\mL^{i}_\lambda\right)^2\quad
\\
&=&-\frac{C_0}{\lambda(\lambda-1)}
(\mathcal{L}^i_\lambda)^{n-1}\cdot\left( \sum_{ j\in I}
a_j(\lambda) E_j\right)^2,
\end{eqnarray*}
where $C_0$ is a positive constant. Then the lemma follows from Lemma \ref{Zariski}.
\end{proof}

\subsubsection{Invariance of DF at
contraction or flip points} \label{Inv-CF}

If $\lambda_{i+1}> 1$, then
by the definition of MMP with scaling, we pick up a $K_{\mX^{i}}$-negative extremal ray $[R]$ in $
{\rm NE}(\mX^{i}/C)$ such that $R\cdot (K_{\mX^i}+\lambda_{i+1}\mathcal{H}^i)=0$.
 we perform a birational
transformation:

\begin{center}
\hspace{8mm} \xymatrix{
  \ar[rr]^{f^i} \ar[dr]_{}
              \mX^{i}  &  &   {\mY^{i}}  \ar[dl]^{}    \\
                &  \mathbb{P}^1,             }
\end{center}
which contracts all curves $R'$ whose classes $[R']$ are in the ray $\mathbb{R}_{>0} [R]$.
There are two cases:
\begin{enumerate}
\item (Divisorial Contraction)
If $f^i$ is a divisorial contraction. Then $\mX^{i+1}=\mY^{i}$.
Since $f^i$ is a $(K_{\mX^{i}}+\lambda_{i+1}\mH^{i})$-trivial morphism
by the definition of the MMP with scaling, we have
$$
K_{\mX^{i}}+\lambda_{i+1}\mH^{i}=(f^i)^*(K_{\mY^{i}}+\lambda_{i+1}\mH^{i+1}),
$$
which implies
$$\mL_{\lambda_{i+1}}^{i}=(f^i)^*\mL_{\lambda_{i+1}}^{i+1}.$$
Then it follows from Definition \ref{defint} and projection formula
that
$$ \Fut(\mX^{i}/C,\mL^{i}_{\lambda_{i+1}})=\Fut(\mX^{i+1}/C,
\mL^{i+1}_{\lambda_{i+1}}).
$$

\item (Flipping Contraction)
If $f^{i}$ is a flipping contraction, let $\phi^{i}:
\mX^{i}\dashrightarrow \mX^{i+1}$ be the flip.

\begin{center}
\hspace{8mm} \xymatrix{
  \mX^{i} \ar@{-->}[rr]^{\phi^{i}} \ar[dr]_{-K_{\mX^{i}} \mbox{ is } f^{i}\mbox{-ample
  }}^{f^{i}}
                &  &   \mX^{i+1}  \ar[dl]^{\;\;K_{\mX^{i+1}} \mbox{ is } f^{i+}\mbox{-ample }}_{f^{i+}}    \\
                & \mY^{i}              }
\end{center}

As $f^i$ is a $K_{\mX^{i}}+\lambda_{i+1}\mH^{i}$-trivial morphism,
$K_{\mX^{i}}+\lambda_{i+1}\mH^{i}=(f^{i})^*D_{\mY^{i}}$ for some
divisor $D_{\mY^{i}}$. Since $f^i$, $f^{i+}, \phi^i$ are
isomorphisms in codimension one, we also have
$K_{\mX^{i+1}}+\lambda_{i+1}\mH^{i+1}=(f^{i+})^*D_{\mY^{i}}$.
Therefore, using the projection formula, we see that
$$
\Fut(\mX^i/C,
K_{\mX^{i}}+\lambda_{i+1}\mathcal{H}^{i})=\Fut(\mY^i/C,
D_{\mY^i})=\Fut(\mX^{i+1}/C,
K_{\mX^{i+1}}+\lambda_{i+1}\mathcal{H}^{i+1}).
$$

\end{enumerate}

Now we can finish the proof of Proposition \ref{P-DFMMP}:
\begin{proof}[Completion of proof of Proposition \ref{P-DFMMP}]
By the discussion in \ref{Non-in} and \ref{Inv-CF}, we have
\begin{eqnarray*}
& &\Fut(\mX^{0}/C, \mL^{0}_{\lambda_0})\ge \Fut(\mX^{0}/C,
\mL^{0}_{\lambda_1})\\
&=&\Fut(\mX^{1}/C, \mL^{1}_{\lambda_1})\ge \Fut(\mX^{1}/C,
\mL^{1}_{\lambda_2})\\
&&\cdots\quad  \cdots \quad \cdots\\
&=&\Fut(\mX^{i}/C, \mL^{i}_{\lambda_i})\ge \Fut(\mX^{i}/C,
\mL^{i}_{\lambda_{i+1}})\\
&=&\Fut(\mX^{i+1}/C, \mL^{i+1}_{\lambda_{i+1}})\ge \Fut(\mX^{i+1}/C,
\mL^{i+1}_{\lambda_{i+2}})\\
&&\cdots\quad \cdots \quad \cdots\\
&=&\Fut(\mX^{k}/C, \mL^{k}_{\lambda_k})=\Fut(\mX^{k}/C,-K_{\mX^k}).
\end{eqnarray*}

 We proceed to characterize the equality case. Since
$-K_{\mX^k}\sim_{\mathbb{Q}, C}\mL^{k}_{\lambda_k}$ is
relatively nef over $C$, we conclude that $f^{k-1}:\mX^{k-1}\to
\mX^k$ can only be a divisorial contraction. Therefore, $h:
\mX^0\dashrightarrow \mX^k$ contracts at least one divisor if it is
not an isomorphism.

Since $\mX^k$ is a minimal model of $(\mX^0,\mH^0)$ (see Theorem \ref{t-mmp}(2)),
we know that
$$0<E=K_{\mX^0}+\mH^0-h^*(K_{\mX^k}+\mH^k)
\sim_{\mathbb{Q},
C}K_{\mX^0}+\mH^0,$$
which is supported on the fiber over $0$. It follows from the fact that the support of $E$
is a proper subset of $\mX^0_0$,
$K_{\mX^{0}}+\mH^{0}$ is not $\mathbb{Q}$-linearly
equivalent to 0 over $C$, i.e.,
 the equality
condition of Lemma \ref{FUT}  can not hold on $\mX^{{\rm lc}}$. Thus a  for
sufficiently small rational number $\epsilon$,
$$\Fut(\mX^{{\rm lc}}/C,\mL^{{\rm lc}})=\Fut(\mX^0/C,\mL_{\lambda_0}^0)>\Fut(\mX^0/C,\mL^0_{\lambda_0-\epsilon})\ge \Fut(\mX^k/C,\mL^k_{\lambda_k}).$$ 
\end{proof}


\section{Revisiting of $\mathbb{Q}$-Fano extension}\label{atpsef}

Let us continue the study of Example \ref{ex-deg}.
\begin{exmp}
We use the notation in Example \ref{ex-deg}.
Since $K_{\mathcal{X}}$ is of bidegree $(-1,-1)$,
$$\Fut(\mX/\mathbb{P}^1,-K_{\mX})=-\frac{1}{2(n+1)(-K_{\mX_t})^n}(-K_{\mX/\mathbb{P}^1})^{n+1}=\frac{4}{9}.$$ 
Using the intersection formula, we easily see 
\begin{eqnarray*}
\Fut(\tilde{\mX}/\mathbb{P}^1,-K_{\tilde{\mX}})&=&\Fut(\tilde{\mX}/\mathbb{P}^1,-(K_{\tilde{\mX}}+\tilde{\mX_0}))\\
&=&3\cdot \Fut(\mX'/\mathbb{P}^1,-(K_{\mX'}+\mX'_0))\\
&=&3\cdot \Fut(\mX/\mathbb{P}^1,-(K_{\mX}+\mX_0))\\
&=&3\cdot \Fut(\mX/\mathbb{P}^1,-K_{\mX})\\
&=&\frac{4}{3}.
\end{eqnarray*}
Since $K_{\tilde\mX}|_{S_1}=(K_{\tilde\mX}+\tilde\mX_0)|_{S_1}$ is trivial, we calculate 
\begin{eqnarray*}
\Fut(\mX^{\rm s}/\mathbb{P}^1,-K_{\mX^{{\rm s}}})&=&\frac{1}{18}(K_{\mX^{{\rm s}}/\mathbb{P}^1})^3\\
&=&\frac{1}{18}(K_{\tilde{\mX}/\mathbb{P}^1}-T)^3=\frac{1}{18}(K_{\tilde{\mX}/\mathbb{P}^1}-\tilde\mX_0+S_1)^3\\
&=&\frac{4}{3}-\frac{1}{18}(9-3)<3\cdot   \Fut(\mX/\mathbb{P}^1,-K_{\mX}).
\end{eqnarray*}
Therefore,
we see that if we  normalize the DF invariants by dividing the degree of the base change, then our process in this example decreases this normalized DF invariant.
\end{exmp}

From the discussion of the last section, we achieve a model
$\mX^{{\rm ac}}$ over $C$ with polarization $\mL^{{\rm ac}}$ which compactifies
$\mX^{*}/C^*$ such that $\mL^{{\rm ac}}\sim_{\mathbb{Q},C}-K_{\mX^{{\rm ac}}}$
and $(\mX^{{\rm ac}}, \mX_t^{{\rm ac}})$ is log canonical for any $t\in C$. We can not run an MMP
directly from $\mX^{{\rm ac}}$ to get $\mX^{{\rm s}}$. Instead we will resolve
$\mX^{{\rm ac}}$ again and run MMP. More precisely, by Theorem
\ref{t-dfano}(2), we know that there exists $\phi: C'\to C$ with  a $\mathbb{Q}$-Fano family
$\mX^{{\rm s}}/C'$. We will show this is our final $\mathbb{Q}$-Fano
family by verifying the decreasing of the DF invariant.

Using the notation in Theorem \ref{t-dfano}, 
$a(\mX^{{\rm s}}_0;\tilde{\mX}^{{\rm ac}})=0$ implies
$$a(\mX^{{\rm s}}_0;\tilde{\mX}^{{\rm ac}},\tilde{\mX}^{{\rm ac}}_0)=-1,$$
since $\tilde{\mX}^{{\rm ac}}_0$ is Cartier and
$(\tilde{\mX}^{{\rm ac}},\tilde{\mX}^{{\rm ac}}_0)$ is log canonical.  Then for any number $\lambda\in[0,1]$, we know that
$$a(\mX^{{\rm s}}_0;\tilde{\mX}^{{\rm ac}},\lambda\tilde{\mX}^{{\rm ac}}_0)=-\lambda.$$
In particular, there exists a model $\pi':\mX'\to \tilde{\mX}^{{\rm ac}}$
which precisely extracts the divisor $\mX^{{\rm s}}_0$ (cf.
\cite[1.4.3]{BCHM}). Since $a(\mX^{{\rm s}}_0;{\tilde{\mX}}^{{\rm ac}})=0$, we
know $\pi'^*(K_{\tilde{\mX}^{{\rm ac}}} )=K_{\mX'}$. Then by the projection formula
$$\Fut(\mX'/C',-K_{\mX'})=
\Fut(\tilde{\mX}^{{\rm ac}}/C',-K_{\tilde{\mX}^{{\rm ac}}})=\deg(\phi)\cdot \Fut(\mX^{{\rm ac}}/C,-K_{\mX^{{\rm ac}}}).$$

\begin{prop}\label{P-MMP}
We have the inequality
$$\Fut(\mX'/C',-K_{\mX'})\ge \Fut(\mX^{{\rm s}}/C',-K_{\mX^{{\rm s}}}),$$
and the equality holds if and only if the rational map
$\tilde{\mathcal{X}}^{{\rm an}}\dashrightarrow \mX^{{\rm s}}$ is an isomorphism.
\end{prop}

\begin{proof}

By abuse of notation, we identify $C$ and $C'$, ${\mX}^{{\rm ac}}$ and
$\tilde{\mX}^{{\rm ac}}$.


Using the intersection formula, we have that
$$\Fut(\mX'/C,-K_{\mX'/C})=-\frac{1}{2(n+1)(-K_{\mX_t})^n}(-K_{\mX'/C})^{n+1}.$$ Similarly,

$$\Fut(\mX^{{\rm s}}/C,-K_{\mX^{{\rm s}}/C})=-\frac{1}{2(n+1)(-K_{\mX_t})^n}(-K_{\mX^{{\rm s}}/C})^{n+1}.$$

Let $p:\hat{\mX}\to \mX'$ and $q:
\hat{\mX}\to \mX^{{\rm s}}$ be a common log resolution, and we write
$$(\pi'\circ p)^*(K_{\mathcal{X}^{{\rm ac}}})=p^*K_{\mX'}=q^*K_{\mX^{{\rm s}}}+E.$$
 Since $\mX'\dashrightarrow
\mX^{{\rm s}}$ is a birational contraction, by negativity lemma (cf.
\cite[3.39]{KM}) we  conclude that $E\ge 0$.
For $0\le \lambda \le 1$, let
$$f(\lambda)=(-p^*K_{\mX'/C}+\lambda E)^{n+1}.$$
Then for any $0\le \lambda\le 1$
 \begin{eqnarray}
\frac{df(\lambda)}{d\lambda} &=& (n+1)E \cdot (-p^*K_{\mX'}+\lambda E)^{n}\nonumber\\
 &=&(n+1)E \cdot (-(1-\lambda)p^*K_{\mX'}-\lambda q^*K_{\mX^{{\rm s}}})^{n} \nonumber\\
&\ge& 0 \nonumber,
\end{eqnarray}
since $-(1-\lambda)p^*K_{\mX'}-\lambda q^*K_{\mX^{{\rm s}}}$ is relatively
nef over $C$. Thus
$$\Fut(\mX'/C',-K_{\mX'})\ge
\Fut(\mX^{{\rm s}}/C',-K_{\mX^{{\rm s}}}).$$
We analyze when the equality holds. If $E=0$, then
$$\mX^{{\rm ac}}\cong {\rm Proj}R(\mX'/C,-K_{\mathcal{X}'/C})\cong {\rm Proj} R(\mX^{{\rm s}}/C,-K_{\mX^{{\rm s}}/C})=\mX^{{\rm s}}.$$  So we may assume that the effective $\mathbb{Q}$-divisor $E$ is not equal to 0.

Next we assume that  $\mathcal{X}^{{\rm ac}}$ is isomorphic to $\mX^{{\rm s}}$ in
codimension 1. Thus for any divisor $D$ on $\mathcal{X}^{{\rm ac}}$,
$$R(\mathcal{X}^{{\rm ac}}/C,D)\cong R(\mX^{{\rm s}}/C,D_{\mX^{{\rm s}}}),$$
where $D_{\mX^{{\rm s}}}$ is the push forward of $D$ to $\mX^{{\rm s}}$. In
particular, if we let $D=-K_{\mathcal{X}^{{\rm ac}}}$, we again have
$$\mathcal{X}^{{\rm ac}}\cong {\rm
Proj}R(\mathcal{X}^{{\rm ac}}/C,-K_{\mathcal{X}^{{\rm ac}}/C})\cong {\rm Proj}
R(\mX^{{\rm s}}/C,-K_{\mX^{{\rm s}}/C})=\mX^{{\rm s}}.$$

So we can assume that $E>0$ and  $\mathcal{X}^{{\rm ac}}$ is not isomorphic to
$\mX^{{\rm s}}$ in codimension 1, then we claim that:
\begin{equation}\label{decvol}
 f(0)<f(\lambda)
\end{equation}
 for any $1>\lambda>0$.

In fact, since  $\mX^{{\rm s}}_0$ is irreducible, from the above discussion
we may assume that there exists a component $E_1\subset \mX^{{\rm ac}}_0$
such that the birational transform $\hat{E}_1$ of $E_1$ on
$\hat{\mX}$ is contracted under $\hat{\mX}\to \mX^{{\rm s}}$. As
$-K_{\mX^{{\rm ac}}}$ is ample on $E_1$, $-(\pi'\circ p)^*K_{\mX^{{\rm ac}}}$ is
nontrivial on the generic fiber of  $\hat{E}_1 \to {\rm
center}_{\mX^{{\rm s}}}(E_1)$. This implies $\hat{E}_1\subset {\rm Supp}(E)$
(cf. \cite[3.39]{KM}). Denote the coefficient of $\hat{E}_1$ in $E$ to
be $a>0$. Then
\begin{eqnarray}
\left.\frac{df(\lambda)}{d\lambda}\right|_{\lambda=0} &=& (n+1)E \cdot (-p^*K_{\mX'})^{n}\nonumber\\
 &\ge&(n+1)a\hat{E}_1 \cdot (-p^*K_{\mX'})^{n}\nonumber\\
 &=& (n+1)aE_1 \cdot (-K_{\mX^{{\rm ac}}})^{n} \nonumber\\
&>& 0 \nonumber.
\end{eqnarray}
Since $f(\lambda)$ is nondecreasing on $\lambda\in [0,1]$ and its derivative at $0$ is positive, we easily see $f(\lambda)> f(0)$ for any $\lambda\in (0,1]$.

\end{proof}


\section{Proof of theorems}\label{PET}

In this section, we finish proving Theorem \ref{main}
by combining the three steps proved in Theorem \ref{t-step1},
Theorem \ref{t-step2} and Theorem \ref{t-step3}.

\begin{proof}[Proof of Theorem \ref{main}]

Let $(\mX,\mL)$ be any polarized generic $\mathbb{Q}$-Fano family. Note that, in particular, we assume $\mX$ is normal. 

Then it follows from Theorem \ref{t-step1} that, after a base change
$\phi: C'\to C$, we get a polarized generic $\mathbb{Q}$-Fano family
$(\mX^{{\rm lc}},\mL^{{\rm lc}})$ satisfying the properties stated in Theorem
\ref{t-step1}. 
Letting $l>0$ be a sufficiently large integer,
we can run a sequence of $K_{\mX^{{\rm lc}}}$-MMP over $C$ with
scaling of $\mH^{{\rm lc}}=\mL^{{\rm lc}}-(l+1)^{-1} (K_{\mX^{{\rm lc}}}+\mL^{\rm lc})$ as in Section
\ref{bfpsef}. We obtain a model $\mX^{k}\to C'$ with $-K_{\mX^k}$ is
relatively big and semi-ample. Therefore, it admits an
anti-canonical model $\mX^{{\rm ac}}$ which satisfies that
$(\mX^{{\rm ac}},\mX^{{\rm ac}}_t)$ is log canonical for every $t\in C'$.
Finally, after a base change $C''\to C'$, we construct a
$\mathbb{Q}$-Fano familly $\mX^{{\rm s}} \to C''$. After base change to
$C''$, all our models after base change are isomorphic over $C^*\times_C C''$.

For the DF invariants,
\begin{eqnarray}\label{decseq2}
 \deg(C'/C)\cdot\Fut(\mX/C, \mL)
&\ge&\Fut(\mX^{{\rm lc}}/C', \mL^{{\rm lc}}) \hspace{5mm} \mbox{(by Theorem \ref{t-step1})}  \label{decnlc} \\
&\ge&\Fut(\mX^{{\rm ac}}/C', -K_{\mX^{{\rm ac}}}) \hspace{5mm} \mbox{(by Theorem \ref{t-step2})} \label{declcan}\\
&\ge&\frac{1}{\deg(C''/C')}\Fut(\mX^{{\rm s}}/C'',-K_{\mX^{{\rm s}}}) \hspace{5mm}
\mbox{(by Theorem \ref{t-step3})} \label{declcm}.
\end{eqnarray}

By Theorem \ref{t-step3}, the equality in \eqref{declcm} holds if
and only if $\tilde{\mX}^{{\rm ac}}=\mX^{{\rm s}}$. Assume that $t''\in C''$ is
mapped to $t'\in C'$, then $(\mX^{{\rm ac}},\mX^{{\rm ac}}_{t'})$ is plt if and
only if $(\tilde{\mX}^{{\rm ac}},\tilde{\mX}^{{\rm ac}}_{t''})$ (see
\cite[5.20]{KM}) which then implies that $\mX^{{\rm ac}}$ is a
$\mathbb{Q}$-Fano family over $C'$.

By Theorem \ref{t-step2}, the equality in \eqref{declcan} holds if and
only if $\mX^{{\rm lc}}=\mX^{{\rm ac}}$ and $\mL^{{\rm lc}}=\mL^{{\rm ac}}$.

Finally by Theorem \ref{t-step1},  the equality in \eqref{decnlc}
holds if and only if $(\mX, \mX_t)$ is log canonical for any $t\in C$ and
$\mX^{{\rm lc}}\cong \mathcal{X}\times_{C}C'$. As
$(\mX^{{\rm lc}},\mX^{{\rm lc}}_t)\cong (\mX^{{\rm s}},\mX^{{\rm s}}_t)$ is a plt for any $t\in C'$ and $\mL^{{\rm lc}}\sim_{\mathbb{Q}, C'}-K_{\mX^{{\rm lc}}}$,  this implies that $(\mX,\mX_t)$ is plt for any $t\in C$ (cf. \cite[5.20]{KM}) and $\mL\sim_{\mathbb{Q}, C}-K_{\mX}$, i.e., $\mX$ is a $\mathbb{Q}$-Fano family over $C$.
\end{proof}

\part{Application to KE Problem}

In this part, we consider the application of  Theorem \ref{main} to the study of K-stability of Fano varieties and existence of K\"abler-Einstein metric.

\section{Introduction: K-stability from K\"{a}hler-Einstein problem}\label{ss-KE}
A fundamental
problem in K\"{a}hler geometry is to determine whether there exists
a K\"{a}hler-Einstein metric on a Fano manifold $X$, i.e., to find a K\"{a}hler
metric $\omega_{{\rm KE}}$ in the K\"{a}hler class $c_1(X)$ satisfying the
equation:
\[
{\rm Ric}(\omega_{{\rm KE}})=\omega_{{\rm KE}}.
\]
This is a variational problem. Futaki \cite{Fut} found an important
invariant as the obstruction to this problem. Then Mabuchi
\cite{Mab} defined the K-energy functional by integrating this
invariant. This is a Kempf-Ness type function on the infinite dimensional space of K\"{a}hler metrics in $c_1(X)$. 
The minimizer of the K-energy is the K\"{a}hler-Einstein
metric. Tian \cite{Tian97} proved that, under some restriction on the automorphism group, there is a
K\"{a}hler-Einstein metric if and only if the K-energy is proper on
the space of all K\"{a}hler metrics in $c_1(X)$. So the problem is
how to test the properness of the K-energy.

Tian also developed a program to reduce this infinite dimensional problem
to finite dimensional problems. More precisely, he proved in
\cite{Tian1} that the space of K\"{a}hler metrics in a fixed
K\"{a}hler class can be approximated by a sequence of spaces
consisting of Bergman metrics. The latter spaces are finite
dimensional symmetric spaces. Tian (\cite{Tian97}) then introduced
the K-stability condition using the generalized Futaki invariant
(\cite{DT}) for testing the properness of K-energy on these finitely
dimensional spaces. Later Donaldson \cite{Dol} reformulated it by defining the Futaki invariants 
algebraically (see \eqref{deffut}), which
is now called  {\it the Donaldson-Futaki invariant}. The following
folklore conjecture is the guiding question in this area.
\begin{conj}[Yau-Tian-Donaldson conjecture]\label{YTD}
Let $X$ be a Fano manifold. Then there is a  K\"{a}hler-Einstein metric in $-c_1(X)$ if and only if $(X,
-K_X)$ is K-polystable. \footnote{In the recent remarkable works, Tian and Chen-Donaldson-Sun announced proofs of this conjecture \cite{Tian12, CDSI, CDSII, CDSIII}. Their results also imply Corollary \ref{Tian} below.}
\end{conj}
See \cite{Tian2} for more detailed discussion for the
K\"{a}hler-Einstein problem. 

In the following we will recall the definition of K-stability. First we need to define the notion of {\it test configuration}.
\begin{defn}\label{defineTC}
1. Let $X$ be an $n$-dimensional $\mathbb{Q}$-Fano variety. Assume that $-rK_X$ is
Cartier for some fixed $r\in \mathbb{N}$. A test configuration of $(X, -rK_X)$ consists of
\begin{enumerate}
\item[ $\cdot$]
a variety $\mX^{{\rm tc}}$ with a $\mathbb{G}_m$-action,
\item[ $\cdot$]
a $\mathbb{G}_m$-equivariant ample line bundle $\mL^{{\rm tc}}\rightarrow
\mX^{{\rm tc}}$,
\item[ $\cdot$]
a flat $\mathbb{G}_m$-equivariant map $\pi: (\mX^{{\rm tc}},
\mathcal{L}^{{\rm tc}})\rightarrow \mathbb{A}^1$, where $\mathbb{G}_m$
acts on $\mathbb{A}^1$ by multiplication in the standard way
$(t,a)\to ta$,
\end{enumerate}
such that for any $t\neq 0$, $(\mX^{{\rm tc}}_t,
\mL^{{\rm tc}}_t)$ is isomorphic to $(X, -rK_X)$, where $\mX^{{\rm tc}}_t=\pi^{-1}(t)$ and $\mL^{{\rm tc}}_t= \mL^{{\rm tc}}|_{\mX^{{\rm tc}}_t}$.

\noindent 2.  Fix $r\in \mathbb{Q}_{>0}$. We call $(\mX^{{\rm tc}}, \mL^{{\rm tc}})$ a $\mathbb{Q}$-test configuration of
$(X, -rK_X)$ if $\mathcal{L}^{{\rm tc}}$ is a $\mathbb{Q}$-Cartier
divisor class on $\mX^{{\rm tc}}$ such that for some integer $m\ge 1$,
$(\mX^{{\rm tc}},m\mL^{{\rm tc}})$ yields a test configuration of $(X, -mrK_X)$.
\end{defn}

Obviously we can rescale the polarization of any test configuration to obtain a $\mathbb{Q}$-test configuration of $(X, -K_X)$.

Similarly to the notion of $\mathbb{Q}$-Fano family, we have the following definition.
\begin{defn}A normal $\mathbb{Q}$-test configuration $(\mathcal{X}^{{\rm tc}},\mathcal{L}^{{\rm tc}})$ of $(X,-K_{X})$ is called {\it a special $\mathbb{Q}$-test configuration} if $\mathcal{L}^{{\rm tc}}\sim_{\mathbb{Q}}-K_{\mX^{{\rm tc}}}$ and $\mX_0^{{\rm tc}}$ is a $\mathbb{Q}$-Fano variety. Whenever there is no ambiguity, we will also abbreviate it as a special test configuration.
\end{defn}

For any $\mathbb{Q}$-test configuration, we can define {\it the Donaldson-Futaki
invariant}. First  by Riemann-Roch theorem,  for sufficiently divisible $k\in\mathbb{N}$, we have
\[
d_k=\dim H^0(X, \mathcal{O}_X(-kK_{X}))=a_0k^n+a_1k^{n-1}+O(k^{n-2})
\]
for some rational numbers $a_0$ and $a_1$. Let
$(\mX_0^{{\rm tc}},\mL_0^{{\rm tc}})$ be the restriction of $(\mX^{{\rm tc}},
\mL^{{\rm tc}})$ over $\{0\}$. Since $\mathbb{G}_m$ acts on
$(\mX_0^{{\rm tc}},\mL^{tc\otimes k}_0)$, $\mathbb{G}_m$ also acts on
$H^0(\mX_0^{{\rm tc}},k\mL^{{\rm tc}}_0)$. We denote the total
weight of this action by $w_k$. By the equivariant Riemann-Roch
Theorem,
\[
w_k=b_0k^{n+1}+b_1k^{n}+O(k^{n-1}).
\]
So we can expand
\[
\frac{w_k}{kd_k}=F_0+F_1k^{-1}+O(k^{-2}).
\]
\begin{defn}[\cite{Dol}]\label{defineDFtc}
The (normalized) Donaldson-Futaki invariant (DF-invariant) of the
$\mathbb{Q}$-test configuration $(\mX^{{\rm tc}}, \mL^{{\rm tc}})$ is defined to be
\begin{equation}\label{deffut}
\Fut(\mX^{{\rm tc}}
,\mL^{{\rm tc}})=-\frac{F_1}{a_0}=\frac{a_1b_0-a_0b_1}{a_0^2}
\end{equation}
\end{defn}

With the normalization in \eqref{deffut}, we easily see for any  $a\in \mathbb{Q}_{>0}$, 
$\Fut(\mX^{{\rm tc}},\mL^{{\rm tc}})=\Fut(\mX^{{\rm tc}},a\mL^{{\rm tc}})$. 

\begin{rem}\label{history}
We have the following remarks about Donaldson-Futaki invariant.
\begin{enumerate}
\item From the differential geometry side, Ding and Tian \cite{DT} defined
the generalized Futaki invariants by extending the original differential
geometric formula of Futaki \cite{Fut} from smooth manifolds to
normal varieties. On a normal $\mathbb{Q}$-Fano variety, the
differential geometric definition coincides with the above algebraic
definition. This was proved by Donaldson \cite{Dol} in the smooth
case. The calculation via equivariant forms in \cite{Dol} is also
valid in the normal case, because the codimension of singularities
on a normal variety is at least two.

\item In \cite{PauTi} and \cite{PauTi2}, Paul and Tian proved that the
Donaldson-Futaki invariant is the same as the total  $\mathbb{G}_m$-weight of the CM line
bundle, which was introduced to give a GIT formulation
of K-stability. See \cite{FS90,Tian97,PauTi,PauTi2}
for details.
\end{enumerate}
\end{rem}

As we will show in Section \ref{int}, by adding a `trivial fiber' over the point $\infty \in \mathbb{P}^1$, we can compactify a
$\mathbb{Q}$-test configuration $(\mX^{{\rm tc}},\mL^{{\rm tc}})\to
\mathbb{A}^1$  to obtain a polarized generic $\mathbb{Q}$-Fano family $(\bX^{{\rm tc}}/\mathbb{P}^1,\bar{\mL}^{{\rm tc}})$. By comparing the DF invariant and DF invariant,  we have the equality
$$\Fut(\bX^{{\rm tc}}/\mathbb{P}^1,\bar{\mL}^{{\rm tc}})=\Fut(\mX^{{\rm tc}},\mL^{{\rm tc}}),$$
which explains the origin of our terminology.

If we apply our Theorem \ref{main} to this case, it specializes to the following result.

\begin{thm}\label{mainTC}
Let $X$ be a $\mathbb{Q}$-Fano variety and  $(\mX^{{\rm tc}},
\mL^{{\rm tc}})$  a test configuration of
$(X,-K_X)$. We can construct a special test configuration
$(\mX^{{\rm st}},-K_{\mX^{{\rm st}}})$ and a positive integer $m$, such that
$$m\Fut(\mX^{{\rm tc}},\mL^{{\rm tc}})\ge
\Fut(\mX^{{\rm st}},-K_{\mX^{{\rm st}}}). $$ Furthermore, if we assume that
$\mX^{{\rm tc}}$ is normal, then the equality holds for our construction only when
$(\mX^{{\rm tc}},\mL^{{\rm tc}})$ itself is a special test configuration.
\end{thm}

This corollary will be applied to study K-stability, which we define in the below.
\begin{defn}\label{d-kstable}
 Let $X$ be a $\mathbb{Q}$-Fano variety.
\begin{enumerate}
\item
$X$ is called K-semistable if for any $\mathbb{Q}$-test
configuration $(\mX^{{\rm tc}}, \mL^{{\rm tc}})$ of $(X, -K_X)$, we
have $\Fut(\mX^{{\rm tc}}, \mL^{{\rm tc}})\ge0$.
\item
$X$ is called K-stable (resp. K-polystable) if for any normal
$\mathbb{Q}$-test configuration $(\mX^{{\rm tc}}, \mL^{{\rm tc}})$ of $(X,
-K_X)$, we have $\Fut(\mX^{{\rm tc}}, \mL^{{\rm tc}})\ge 0$, and
the equality holds only if $(\mX^{{\rm tc}}, \mL^{{\rm tc}})$ is trivial (resp. only if
$\mX^{{\rm tc}}\cong X\times \mathbb{A}^1$).
\end{enumerate}
\end{defn}

\begin{rem} We have the following remarks for the definitions of K-stability.
\begin{enumerate}
\item Though the notions of K-stability can be stated for a general singular variety $X$ with $-K_X$ being $\mathbb{Q}$-Cartier and ample, in \cite{Oda1}, Odaka shows that for such a variety, if it is K-semistable, it can only have klt singularities.
\item In the definitions of K-polystability and K-stability, for the triviality of the test configuration with Donaldson-Futaki invariant 0, we require the test configuration to be normal.  This is slightly different with the original definition. However, we believe this should be the right one. For more details, see Section \ref{s-ntc}.  It is a consequence of \cite[5.2]{RT} that we only need to consider normal test configurations for $K$-semistability, too.
\end{enumerate}
\end{rem}

All these notions of test configurations, Donaldson-Futaki invariants and K-(semi,poly)-stability can be defined for a general polarized projective variety $(X,L)$. A more general version of Conjecture \ref{YTD}  predicts the equivalence between the K-polystability of a polarized manifold  $(X,L)$ and the existence of a constant scalar curvature K\"abler metric in $c_1(L)$. 
Nevertheless, in this paper, except in Section \ref{s-ntc}  we mainly consider the notion of K-stability for the K\"{a}hler-Einstein problem on Fano varieties.

In the work \cite{Tian97} where Tian gave the  original definition
of the K-stability in the analytic setting, he only
considered test configurations with normal central fibers. Later he
conjectured that (see \cite{Tianp}), for Fano manifolds, even with Donaldson's
definition, one still only needs to consider those test
configurations with normal central fibers. This is motivated by
compactness results for K\"{a}hler-Einstein manifolds (See
\cite{CCT}). 

As an immediate consequence of Corollary \ref{mainTC}, we verify Tian's conjecture. In fact, it suffices to consider an even smaller class of test configurations, namely, the test configurations whose central fibers are $\mathbb{Q}$-Fano.

\begin{cor}[Tian's conjecture]\label{Tian}
Assume that $X$ is a $\mathbb{Q}$-Fano variety. If $X$ is destablized by
a test configuration, then $X$ is indeed destablized by a special
test configuration. More precisely, the following two statements are
true.
\begin{enumerate}
\item(K-semistability) If $X$ is not K-semi-stable, then there exists a  special $\mathbb{Q}$-test configuration
$(\mX^{{\rm st}}, -K_{\mX^{{\rm st}}})$ with a negative Futaki invariant
$\Fut(\mX^{{\rm st}},-K_{\mX^{{\rm st}}})<0$.
\item (K-polystability) Let $X$ be a K-semistable variety. If $X$ is not K-polystable, then there exists a special $\mathbb{Q}$-test configuration $(\mathcal{X}^{{\rm st}},-K_{\mX^{{\rm st}}})$ with Donaldson-Futaki invariant 0 such that
 $\mX^{{\rm st}}$
is not isomorphic to $X\times\mathbb{A}^1$.
\end{enumerate}
\end{cor}


\section{Donaldson-Futaki invariant and K-stability}\label{DF}

In this section, we will concentrate on the study of Donaldson-Futaki invariants of a test configuration, which is algebraically defined by Donaldson \cite{Dol}. In the first subsection, we recall the fact that for a  given test configuration $(\mX^{{\rm tc}},\mL^{{\rm tc}})\to \mathbb{A}^1$, its Donaldson-Futaki invariant coincides with the Donaldson-Futaki invariant of the natural compactification  $(\bar {\mX}^{{\rm tc}},\bar{\mL}^{{\rm tc}})\to \mathbb{P}^1$.
This characterization of the Donaldson-Futaki invariants first appears in Wang's work \cite{Wang}. A different proof was also given in \cite{Oda}. In the second subsection, we correct a small inaccuracy in the original definition of K-polystability in literatures.



\subsection{Intersection formula for the Donaldson-Futaki invariant}\label{int}
Given any test configuration $(\mX^{{\rm tc}}, \mL^{{\rm tc}})$, we first
compactify it by gluing $(\mX^{{\rm tc}}, \mL^{{\rm tc}})$ with $(X
\times (\mathbb{P}^1\setminus \{0\}) ,p_1^*L)$. It is known that the
Donaldson-Futaki invariant is equal to the DF invariant on this
compactified space as defined in Definition \ref{defint} (see \cite{Wang,Oda}).  We will present a proof for reader's convenience.

\begin{exmp}
$\mathbb{G}_m$ acts on $(X,L^{-1})=(\mathbb{P}^1,
\mathcal{O}_{\mathbb{P}^1}(-1))$ by
$$t\circ ([Z_0, Z_1], \lambda (Z_0, Z_1))=([Z_0, t Z_1], \lambda (Z_0, t Z_1)).$$ In particular,
the $\mathbb{G}_m$-weights on
$$\mot|_0, \mop|_0, \mot|_\infty \mbox{ and } \mop|_\infty$$
are 0,0,1 and -1. Let $\tau_0=Z_1$, $\tau_\infty=Z_0$ be the
holomorphic sections of $\mop$. Then the $\mathbb{G}_m$-weights of
$\tau_0$ and $\tau_\infty$ are $-1$ and $0$.

Take $\bX=\mathbb{P}(\mop\oplus\mathcal{O}_{\mathbb{P}^1})$ and
$\bL=\mathcal{O}_{\bX}(1)=\mathcal{O}_{D_\infty}$, where $D_\infty$
is the divisor at infinity.  We see that $(\mX^{{\rm tc}}:=\bX\setminus
\mathbb{P}^1_{\infty},\mL^{{\rm tc}}:=\bL|_{\mX^{{\rm tc}}})$ yields a test
configuration of $(X,L)$. Then $H^0(\mathbb{P}^1, L^{\otimes k})$ is
of dimension $d_k=k+1$ and by the calculation in the first paragraph
the total $\mathbb{G}_m$-weight of  $H^0(\mathbb{P}^1, L^{\otimes
k})$ is $w_k=-\frac{1}{2}(k^2+k)$. We  know $D_\infty^2=-1$ and
$K_{\bX}^{-1}\cdot D_\infty=1$. So
$$
w_k=\frac{D_\infty^2}{2}k^2+\left(\frac{K_{\bX}^{-1}\cdot D_\infty
}{2}-1\right)k\qquad \mbox{and}\qquad
\Fut(\mX^{{\rm tc}},\mL^{{\rm tc}})=\frac{D_\infty^2}{2}-\left(\frac{K_{\bX}^{-1}\cdot
D_\infty }{2}-1\right)(=0).
$$

\end{exmp}

This example illustrates general cases (see
\eqref{b0}, \eqref{b1}). In the following we will use Donaldson's argument
(see the proof of Proposition 4.2.1 in
\cite{Dol}) to get the general intersection formula for DF invariant. 

First note that, after identifying the fiber $\mX^{{\rm tc}}_1$ over
$\{1\}$ and $X$, we have an equivariant isomorphism:
\[
(\mX^{{\rm tc}}\backslash\mX^{{\rm tc}}_0, \mL^{{\rm tc}})\cong (X \times (
\mathbb{A}^1\setminus \{0\}), p_1^*L)
\]by $(p,a,s)\to (a^{-1}\circ p,a,a^{-1}\circ s)$.
Therefore, $\mathbb{G}_m$ acts on the right hand side by
\[
t\circ (\{p\}\times \{a\}, s)=(\{ p\} \times \{ta\},  s)
\]
for any $p\in X$, $a\in \mathbb{A}^1$ and $s\in\mL^{{\rm tc}}_p$. The
gluing map is given by
\[
\begin{array}{ccc}
(\mX^{{\rm tc}}, \mL^{{\rm tc}})&&( X\times \mathbb{P}^1\setminus\{0\},p_1^*L)
\\
\bigcup&&\bigcup\\
(\mX^{{\rm tc}}\backslash\mX^{{\rm tc}}_0, \mL^{{\rm tc}})
&\longrightarrow& (X\times (\mathbb{A}^1\setminus\{0\}), p_1^*L)
\\
&&\\
(p,a,s)&\longmapsto& (\{ a^{-1}\circ p\}\times \{ a\}, a^{-1}\circ
s),
\end{array}
\]
where $\mathbb{G}_m$ only acts by multiplication on the factor $
\mathbb{P}^1\setminus\{0\}$ of $( X\times
\mathbb{P}^1\setminus\{0\},p_1^*L)$.

\begin{defn}Using the above gluing map, from a test configuration $({\mX}^{{\rm tc}},\mL^{{\rm tc}})$ of $(X,-rK_X)$, we get a generic $\mathbb{Q}$-Fano family
$\bar{\pi}:
(\bar{\mX}^{{\rm tc}},\frac{1}{r}\bar{\mL}^{{\rm tc}})\rightarrow \mathbb{P}^1$, which we call {\it $\infty$-trivial compactification} of the test configuration.
\end{defn}
 In the
following of this subsection, we will denote
$(\bar{\mX}^{{\rm tc}},\bar{\mL}^{{\rm tc}})$ by $(\bX,\bL)$ for simplicity. Note
that there exists an integer $N$, such that
$\bar{\mathcal{M}}=\bar{\mathcal{L}}\otimes
\bar{\pi}^*(\mathcal{O}_{\mathbb{P}^1}(N\cdot \{\infty\}))$ is ample
on $\bX$ (cf. \cite[1.45]{KM}).

We need the following weak form of the Riemann-Roch formula whose
proof is well known.
\begin{lem} Let $X$ be an $n$-dimensional normal projective variety and $L$ an ample divisor
on $X$ then
\[
\dim H^0(X, L^{\otimes
k})=\frac{L^n}{n!}k^n+\frac{1}{2}\frac{(-K_{X})\cdot
L^{n-1}}{(n-1)!}k^{n-1}+O(k^{n-2}).
\]
\end{lem}

We define
\[
d_k=\dim H^0(X, L^{\otimes k})=:a_0 k^{n}+a_1 k^{n-1}+O(k^{n-2})
\]

\begin{prop}\label{Intfor}
Let $(\mX^{{\rm tc}}, \mL^{{\rm tc}})$ be a test configuration of $(X,-rK_X)$. Assume that $\mX^{{\rm tc}}$ is normal, then
\begin{equation}\label{itft}
\Fut(\mX^{{\rm tc}}, \mL^{{\rm tc}})=\Fut(\bX/\mathbb{P}^1,\bL)
\end{equation}
\end{prop}
\begin{proof}
For $k\gg 0$, by Serre Vanishing Theorem, we have two exact
sequences:
\[
\begin{array}{rcccccl}
&A&&B&&C&\\
&\parallel&&\parallel&&\parallel&\\
 0\longrightarrow& H^0(\bX,\bM^{\otimes k}(-\mX^{{\rm tc}}_0))&\stackrel{\otimes
\sigma_0}{\longrightarrow}& H^0(\bX,\bM^{\otimes
k})&\longrightarrow& H^0(\bX_0,\bM^{\otimes
k}|_{\mX^{{\rm tc}}_0})&\longrightarrow 0
\\
&&&&&&\\
0\longrightarrow& H^0(\bX,\bM^{\otimes
k}(-\bX_\infty))&\stackrel{\otimes\sigma_{\infty}}{\longrightarrow}&
H^0(\bX,\bM^{\otimes k})&\longrightarrow&
H^0(\bX_{\infty},\bM^{\otimes k}|_{\bX_\infty})&\longrightarrow 0,\\
&\parallel&&\parallel&&\parallel&\\
&A&&B&&D&
\end{array}
\]
where $\sigma_0, \sigma_\infty$ are sections of
$\bPi^*\mathcal{O}_{\mathbb{P}^1}(1)$ which are pull back of the
divisors $\{0\}, \{\infty\}$ on $\mathbb{P}^1$.

We can assume that the $\mathbb{G}_m$-weights of $\sigma_0$ and
$\sigma_\infty$ are $-1$ and 0. Note the first terms in the two exact
sequences are the same as $A:=H^0(\bX,\bM^{\otimes
k}\otimes\bPi^*\mathcal{O}_{\mathbb{P}^1}(-1))$. We have the
equation: $w_B=w_A-d_A+w_C=w_A+w_D$, where we denote $d_A$ and $w_A$
to be the dimension and the $\mathbb{G}_m$-weight of the vector
space $A$ and similarly for $d_B$, $w_C$ etc. Since the
$\mathbb{G}_m$-weight of $\mathcal{O}_{\mathbb{P}^1}(1)|_{\infty}$
is -1 and $\mathbb{G}_m$ acts on $\bL|_{\bX_\infty}$ trivially, we
have $w_D=-kN \mbox{dim} H^0(\bX_\infty,\bL^{\otimes
k}|_{\bX_\infty})$. So we get
\[ w_C=d_A+w_D=d_B-d_C-kN d_D=d_B-(kN+1)d_C,
\]
Since $\mathbb{G}_m$ acts trivially on
$\mathcal{O}_{\mathbb{P}^1}(1)|_0$, we get the $\mathbb{G}_m$-weight
on $H^0(\mX^{{\rm tc}}_0,\bM^{\otimes k}|_{\mX^{{\rm tc}}_0})=H^0(\mX^{{\rm tc}}_0,
\mL^{tc\otimes k}|_{\mX^{{\rm tc}}_0})$:
\[
w_k=\dim H^0(\bX,\bM^{\otimes k})- (kN+1)\dim
H^0(\mX^{{\rm tc}}_0,\mL^{tc\otimes k}|_{\mX^{{\rm tc}}_0}).
\]
Expanding $w_k$, we get:
\[
w_k=b_0 k^{n+1}+b_1 k^{n}+O(k^{n-1})
\]
with

\begin{equation}\label{b0}
b_0=\frac{\bM^{n+1}}{(n+1)!}-Na_0=\frac{\bL^{n+1}}{(n+1)!}, \mbox{
and }
\end{equation}
\begin{equation}\label{b1}
b_1=\frac{1}{2}\frac{(-K_{\bX})\cdot\bM^{n}}{n!}-N a_1-a_0=
\frac{1}{2}\frac{(-K_{\bX})\cdot\bL^{n}}{n!}-a_0.
\end{equation}
 By substituting the coefficients into
\eqref{deffut}, we get
\begin{eqnarray*}
\frac{a_1b_0-a_0b_1}{a_0^2}&= &\frac{1}{(n+1)!a_0}\left(\frac{a_1}{a_0}\bL^{n+1}+\frac{n+1}{2}K_{\bar{\mX}}\cdot\bL^{n}\right)+1\\
&=&\frac{1}{(n+1)r^n(-K_X)^n}\left(\frac{n}{2r}\bL^{n+1}+\frac{n+1}{2}K_{\bar{\mX}}\cdot\bL^{n}\right)+1\\
&=&\frac{1}{2(n+1)(-K_X)^n}\left(n (\frac{1}{r}\bL)^{n+1}+(n+1)K_{\bar{\mX}/\mathbb{P}^1}\cdot(\frac{1}{r}\bL)^{n}\right)\\
&=& \Fut(\bar\mX/\mathbb{P}^1,\bar\mL).
\end{eqnarray*}
\end{proof}

\begin{rem}\begin{enumerate}
\item
As the above proof shows, Donaldson's formula of Futaki invariant in
the toric case (Proposition 4.2.1 in \cite{Dol}) is a special
example of the intersection formula.
This intersection formula is also related to the interpretation of
Donaldson-Futaki invariant as the CM-weight in \cite{PauTi}. 
\item
When $(\mX^{{\rm tc}},\mathcal{L}^{{\rm tc}})\to \mathbb{A}^1$ is a test
configuration, where we only assume $\mL^{{\rm tc}}$ to be relative big
and semi-ample $\mathbb{Q}$-line bundle, this definition of
Donaldson-Futaki invariant using intersection numbers
$\Fut(\bX^{{\rm tc}}/\mathbb{P}^1,\bL^{{\rm tc}})$  still coincides with
the definition via computing the $\mathbb{G}_m$-weights of
cohomological groups as in \cite{ALV}. For more details, see
\cite{RT} and \cite{ALV}.
\end{enumerate}
\end{rem}

\subsection{Normal test configuration}\label{s-ntc}
It follows from \cite[5.1]{RT} or \cite[3.9]{ALV} that if $n:(\mX',n^*\mL)\to (\mX,\mL)$ is a finite morphism between test configurations of a polarized projective variety $(X,L)$ which is an isomorphism over $\mathbb{C}^*$, then
$$ \Fut(\mX',\mL')\le  \Fut (\mX,\mL),  $$
and the equality holds if and only if $n$ is an isomorphism in codimension 1. 
So to make the definition of K-polystablity reasonable, we need to identify the test configurations which are isomorphic in codimension 1. Alternatively, when $X$ is normal (resp. $S_2$), we could only consider test configurations $\mX$ which are normal (resp. $S_2$) as in Definition \ref{d-kstable}.

In below, we give an easy example which shows that pathological nontrivial test configurations, which are trivial in codimension 1, do occur: 
\begin{exmp} Let $(X, L)=(\mathbb{P}^1, \mathcal{O}_{\mathbb{P}^1}(3))$. Consider the test configuration 
$$\mathcal{X} \subset \mathbb{P}^3\times \mathbb{A}^1=\mathbb{P}(x,y,z,w)\times {\rm Spec}\  k[a]$$ 
given by
$$I=(a^2(x+w)w-z^2,ax(x+w)-yz,xz-ayw,y^2w-x^2(x+w)).$$
(cf. \cite[III.9.8.4]{Ha}). The $\mathbb{G}_m$ action on it is just
sending
$$\mathcal{X}\times \mathbb{G}_m\to \mathcal{X}: (x,y,z,w;a)\times\{ t\} \to (x,y,tz,w;at).$$
Then for $a=0$, the special fiber has the ideal
$$I_0=(z^2,yz,xz,y^2w-x^2(x+w)).$$ Geometrically, $\mX^{{\rm tc}}_t$ is a cubic curve in $\mathbb{P}^3$. They degenerate to the special fiber $\mX^{{\rm tc}}_0$
which is a plane nodal cubic curve in
$\mathbb{P}^2=\mathbb{P}(x,y,w)$ with an embedded point at
$(0,0,0,1)$.

For $k\gg 0$, we have
$$h^0(\mathbb{P}^1,kL)=h^0(\mathbb{P}^1,\mathcal{O}_{\mathbb{P}^1}(3k))=3k+1,\qquad \mbox{and}\qquad H^0(\mathcal{X}_0, k\mL_0)=V_1\oplus V_2,$$ where
$$V_1\cong
H^0(\mX_0^{\rm red}, \mathcal{O}_{\mathbb{P}(x,y,w)}(k)|_{\mX_0^{\rm
red}})$$ and $V_2$ is the one dimensional space spanned by $z\cdot
w^{k-1}$ (or $z\cdot f(x,y,w)$ for any homogeneous polynomial of
degree $k-1$ such that $f(0,0,1)\neq 0$). As the total weight of
$V_1$ is 0 and the total weight of $V_2$ is 1, we conclude that
$b_0=b_1=0$.

\end{exmp}

In fact, it is easy to see in general such pathological test configuration exists for any polarized variety $(X,L)$.

\begin{rem}
\begin{enumerate}
\item 
  In \cite{St}, J. Stoppa claimed a proof of the K-stability of varieties with K\"ahler-Einstein metric  under the original definition, namely without assuming the normality of the test configuration. However, he made a mistake on the calculation of the Donaldson-Futaki invariant for the `degenerate case'.
More precisely, the formula (3.7) of the proof of Proposition 3.3 in
\cite{St} is false because multiplying sections $H^0(\mX^{\rm
red}_0,\mathcal{L}^{k-1}|_{\mX^{\rm red}_0})$ by a nilpotent element
is not always an injection in general. 


There is also a similar overlooking in \cite{PauTi2}. Corollary 2
there said that the properness of K-energy implies Donaldson-Futaki
invariant is positive for any test configuration. The case missing
in their proof in section 3.3 is when $\lim_{t\rightarrow 0}{\rm
Osc}(\phi_t)<\infty$ and the central fibre is generically reduced as
the above example shows.

\item As far as we can see, in most of the published literature, including \cite{RT}, \cite{Oda1} and \cite{Oda2}, the same arguments of proving the
results
on K-stability for certain classes of varieties work, once we
replace the definition there by our new definition. More precisely,
for any nontrivial {\it normal} test configuration $\mX^{{\rm tc}}$, there
is a semi-test configuration $\mY^{{\rm tc}}$ with equivariant morphisms
$p:\mY^{{\rm tc}}\to \mX^{{\rm tc}}$ and $q:\mY^{{\rm tc}}\to X\times \mathbb{A}^1$
such that $q$ is not the trivial morphism. Therefore, $q$ gives an
exceptional divisor $E$ over $X\times \mathbb{A}^1$. Then their
calculations can be carried out by using this exceptional divisor.
\end{enumerate}
\end{rem}

\section{Proof of Theorem \ref{mainTC}}

To prove Theorem \ref{mainTC}, now we only need to check that if we start with a $\infty$-trivial compactification of a test configuration as in Section \ref{int},  the families in theorems \ref{t-step1}, \ref{t-step2} and \ref{t-step3} can be also constructed to be $\infty$-trivial compactifications of test configurations.   

\subsection{Equivariant Semi-stable reduction}

The following result, whose proof is a simple combination of the equivariant resolution (see e.g. \cite[3.9.1]{Ko} and reference therein) and the semistable reducetion (see \cite{KKMS}), is well known. However, we can not find it in the literature. Hence we include a short argument here. 
A similar statement  for resolution appears in \cite{ALV}.
\begin{lem}\label{ssr}
Let $f:\mX\to \mathbb{A}^1$ be a dominant morphism from a normal
variety with a $\mathbb{G}_m$-action such that $f$ is
$\mathbb{G}_m$-equivariant. Then there exists a base change $z^m:
\mathbb{A}^1\to \mathbb{A}^1$ and a semistable family $\mY$ over
$\mathbb{A}^1$ with a morphism $\pi: \mathcal{Y}\to \mX\times_{\mathbb{A}^1, z^m}\mathbb{A}^1
$ which is a log resolution of $(\tilde{\mX},\tilde{\mX}_0)$ where
$\tilde{\mX}$ is the normalization of  $\mX\times_{\mathbb{A}^1,z^m}\mathbb{A}^1$.
\end{lem}

\begin{proof}
 First, we perform the blow up of
$(\mX,\mX_0)$ $\mathbb{G}_m$-equivariantly to get an equivariant log
resolution $\mathcal{Y}^*$. This is always possible by the
theorem of equivariant resolution of singularities.
Then we can take a base change $z^m:\mathbb{A}^1\to \mathbb{A}^1$,
such that the normalization $\tilde{\mathcal{Y}}^{*}$ of $
\mathcal{Y}^*\times_{\mathbb{A}^1,z^m}\mathbb{A}^1$ has a reduced fiber over each point of
$\mathbb{A}^1$. Then it follows from \cite{KKMS} that possibly after a further base change, we can take a
sequence of toroidal blow-ups of $\tilde{\mathcal{Y}}^{*}$ to obtain
a log resolution $\mathcal{Y}$ of
$(\tilde{\mathcal{Y}}^{*},\tilde{\mathcal{Y}}^{*}_0)$ such that $\mathcal{Y}_0$ has reduced
fibers. 

As each component of $\tilde{\mathcal{Y}}^{*}_0$ is $\mathbb{G}_m$-invariant, so are the irreducible components of their intersections.  Since the centers of the toroidal blow-ups are  $\mathbb{G}_m$-invariant,  $\mathbb{G}_m$ action on $\tilde{\mathcal{Y}}^{*}$ can be sucessively lifted to $\mathcal{Y}$.
\end{proof}

\subsection{Proof of Tian's conjecture}
\begin{proof}[Proof of Theorem \ref{mainTC}] It suffices to check that if we start with an $\infty$-trivial compactfication $$(\bar{\mX}^{{\rm tc}},\bar{\mL}^{{\rm tc}})\to \mathbb{P}^1$$ of a $\mathbb{Q}$-test configuration $(\mX^{{\rm tc}},\mL^{{\rm tc}})\to \mathbb{A}^1$, the models we construct in Theorem \ref{t-step1}, \ref{t-step2} and \ref{t-step3} are all $\infty$-trivial compacitifcations of $\mathbb{Q}$-test configurations.

Since resolution of singularities and semistable reduction can be obtained $\mathbb{G}_m$-equivariantly (see Lemma \ref{ssr}), starting from an $\infty$-trivial compactfication $(\bar{\mX}^{{\rm tc}},\bar{\mL}^{{\rm tc}})\to \mathbb{P}^1$, we can perform a base change $d:\mathbb{P}^1\to \mathbb{P}^1$ such that $\bar{\mX}^{{\rm tc}}_d:=\bar{\mX^{{\rm tc}}}\times_{d,\mathbb{P}^1}\mathbb{P}^1\to \mathbb{P}^1$ via the second projection admits a $\mathbb{G}_m$-equivariant semi-stable reduction $\pi:\mY\to \bar{\mX}^{{\rm tc}}_d$.

Then the log canonical modification 
$$\mX^{{\rm lc}}={\rm Proj} R(\mathcal{Y}/\bar{\mX}^{{\rm tc}}_d, K_{\mathcal{Y}})$$
admits a $\mathbb{G}_m$-action such that $\pi^{{\rm lc}}:\mX^{{\rm lc}}_d\to \bar{\mX}^{{\rm tc}}_d$, which is isomorphism over $\mathbb{P}^1\setminus \{0\}$, is equivariant.  Thus the polarization 
$$\mL^{{\rm lc}}=(\pi^{{\rm lc}})^*d^*\bar{\mL}^{{\rm tc}}+tK_{\mX^{{\rm lc}}}$$
for sufficiently small $t>0$ also clearly admits a compatible  $\mathbb{G}_m$-action.

Now we run a relative $K_{\mX^{{\rm lc}}}$-MMP with scaling of $\mL^{{\rm lc}}$ over $\mathbb{P}^1$.  Each step is indeed automatically $\mathbb{G}_m$-equivariant.
In fact, assuming this is true after the $i$-th step, since $\mathbb{G}_m$ is connected, then ${\rm NE}(X^i)^{\mathbb{G}_m}={\rm NE}(X^i)$ (see the proof of \cite[1.5]{Ad}). Hence the contraction is $\mathbb{G}_m$-equivariant. As the flip is a ${\rm Proj}$ of a $\mathbb{G}_m$-equivariant algebra, it also admits a $\mathbb{G}_m$-action. Therefore, at the end $\mX^{{\rm ac}}$ is $\mathbb{G}_m$-equivariant. So $(\mX^{{\rm ac}},-K_{\mX^{{\rm ac}}})$ is a $\infty$-trivial compactfication of the associated test configuration.

Similarly the process involved in the proof of Theorem \ref{t-step3} can be proceeded $\mathbb{G}_m$-equivariantly so that  the special $\mathbb{Q}$-Fano family we obtain at the end yields a special test configuration. We leave the details to the reader.
\end{proof}

\begin{proof}[Proof of Corollary \ref{Tian}]
The semistability case follows from Theorem \ref{main} immediately.

Now we assume that $(X,-rK_X)$ is K-semistable and all special test
configurations $(\mX^{{\rm st}},\mL^{{\rm st}})$ with $\Fut(\mX^{{\rm st}},
\mL^{{\rm st}})=0$ satisfy that $\mX^{{\rm st}}\cong X\times \mathbb{A}^1$. Let
$(\mX^{{\rm tc}},\mL^{{\rm tc}})$ be an arbitrary normal test configuration whose DF invariant is 0. Applying Theorem \ref{main}, we obtain a special test configuration $(\mX^{{\rm st}},-K_{\mX^{{\rm st}}})$ of $(X,-rK_X)$ and inequalities
$$0\le
\Fut(\mX^{{\rm st}},-rK_{\mX^{{\rm st}}})\le m\Fut(\mX^{{\rm tc}}, \mL^{{\rm tc}})=0.$$ Then
since the equality  holds,  by the conclusion of Theorem \ref{main} we know that
$(\mX^{{\rm tc}},\mL^{{\rm tc}})$ is a special test configuration, which implies
that $\mX^{{\rm tc}}\cong X\times \mathbb{A}^1$.
\end{proof}

We finish our paper by the following remark.
\begin{rem}We are inspired by Odaka's algebraic proof of K-stability of canonically polarized variety and Calabi-Yau variety \cite{Oda2}, which provides the counterpart of our theory  for the case when $K_X$ is ample or trivial.
\end{rem}

\textbf{Acknowledgement}: We learnt the inspiring conjecture, which
is answered affirmatively  in Corollary \ref{Tian}, from Professor
Gang Tian. We would like to thank him for many helpful comments and
constant encouragement. The second author is indebted to Yuji Odaka
for his inspiring work \cite{Oda1} and useful discussions, to
Christopher Hacon on his crucial suggestion on proving the results
in Subsection \ref{sub-c}, to Dan Abramovich and J\'anos Koll\'ar on their helpful remarks. We also would like to thank Burt Totaro, Richard
Thomas, Xiaowei Wang and the anonymous referees for their comments which help us to improve the exposition.
 Both authors thank
the Mathematics Institute at Oberwolfach where the joint work started. The second author is partially supported by the Chinese grant `The Recruitment Program of Global Experts'.

\noindent
Department of Mathematics, Princeton University, Princeton, NJ 08544, USA

\noindent 
Current address: Mathematics Department, Stony Brook University, Stony Brook NY, 11794, USA

\noindent
{\sl Email address: chi.li@stonybrook.edu}
\vspace{8mm}

\noindent 
Beijing International Center of Mathematics Research, 5 Yiheyuan
Road, Haidian District, Beijing, 100871, China

\noindent 
{\sl E-mail address: cyxu@math.pku.edu.cn}

\end{document}